%% file: art157.tex
\let\ORIlabel\label
\let\ORIrefstepcounter\refstepcounter
\AddToHook{package/hyperref/before}{%
\let\label\ORIlabel
\let\refstepcounter\ORIrefstepcounter
}
\documentclass[onefignum,onetabnum]{siamart190516}
\usepackage{amsmath,amssymb,color}
\usepackage[mathlines]{lineno}
\usepackage{graphicx,epsfig,framed}
\usepackage{mathptmx,times} % if new font selection scheme installed
\usepackage{epstopdf,textcase} % flushend.sty not reachable  
\usepackage{cite,url} % stfloats.sty not reachable

\usepackage{soul,multirow,pifont} % colortbl.sty not reachable
\usepackage{color,alltt}          % pbox.sty not reachable
\usepackage{enumerate,siunitx}    % breakurl.sty not reachable
\usepackage{psfrag}
\usepackage{float}
\usepackage{ifthen}
\usepackage[numbers,square,sort&compress]{natbib}
\numberwithin{equation}{section}

\newcommand{\rfb}[1]{\mbox{\rm
   (\ref{#1})}\ifx\undefined\stillediting\else:\fbox{$#1$}\fi}

\newfont{\roma}{cmr10 scaled 1200}
\renewcommand{\cline}{{\mathbb C}}
\newcommand{\fline}  {{\mathbb F}}
\newcommand{\nline}  {{\mathbb N}}
\newcommand{\rline}  {{\mathbb R}}
\newcommand{\tline}  {{\mathbb T}}
\newcommand{\GGG} {{\mathbf G}}

\newcommand{\PPP} {{\mathbf P}}
\newcommand{\SSS} {{\mathbf S}}

\newcommand{\dd}  {{\rm d}\hbox{\hskip 0.5pt}}
\renewcommand{\leq} {\leqslant}
\renewcommand{\geq} {\geqslant}

\newcommand{\Dscr} {{\mathcal D}}

\newcommand{\Hscr} {{\mathcal H}}
\newcommand{\Lscr} {{\mathcal L}}

\newcommand{\Nscr} {{\mathcal N}}
\newcommand{\Sscr} {{\mathcal S}}
\newcommand{\Uscr} {{\mathcal U}}
\newcommand{\Yscr} {{\mathcal Y}}
\newcommand{\Rscr}{\mathfrak{R}}
\newcommand{\mm}    {{\hbox{\hskip 0.5pt}}}
\newcommand{\m}     {{\hbox{\hskip 1pt}}}
\newcommand{\n}     {{\hbox{\hskip -5pt}}}
\newcommand{\nm}    {{\hbox{\hskip -3pt}}}

\newcommand{\bluff} {{\hbox{\raise 15pt \hbox{\mm}}}}
\newcommand{\bigbluff} {{\hbox{\raise 21pt \hbox{\mm}}}}
\newcommand{\bbigbluff} {{\hbox{\raise 35pt \hbox{\mm}}}}
\newcommand{\sbluff}{{\hbox{\raise 10pt \hbox{\mm}}}}
\newcommand{\Om}    {{\Omega}}

\newcommand{\e}      {{\varepsilon}}

\renewcommand{\o}    {{\omega}}
\newcommand{\FORALL} {{\hbox{$\hskip 11mm \forall \;$}}}
\newcommand{\rarrow} {\mathop{\rightarrow}}

\newcommand{\half} {{\frac{1}{2}}}
\newcommand{\LE}[1]    {{L^2([0,\infty);#1)}}
\newcommand{\LEloc}[1] {{L^2_{\rm loc}([0,\infty);#1)}}
\newcommand{\NL}{{\rm NL}}
\renewcommand{\theenumi}{\roman{enumi}}

\makeatletter
\renewcommand{\p@enumii}{}
\makeatother

\newcommand{\ULax}   {\boldsymbol{{\mathfrak{T}}}}
\newcommand{\GothA}  {\boldsymbol{{\mathfrak{A}}}}

\newcommand{\loc}{{\rm loc}}

\newcommand{\bbm}[1]{\left[\begin{matrix} #1 \end{matrix}\right]}
\newcommand{\sbm}[1]{\left[\begin{smallmatrix} #1
\end{smallmatrix}\right]}

\input{ex_shared}

\ifpdf
\hypersetup{
  pdftitle={Second order systems on Hilbert spaces \\
  with nonlinear damping},
  pdfauthor={Shantanu. Singh and George. Weiss}}
\fi

\begin{document}

\maketitle

\begin{abstract}
We investigate a special class of nonlinear infinite dimensional
systems. These systems are obtained by modifying the second order
differential equation that is part of the description of conservative
linear systems ``out of thin air" introduced by M. Tucsnak and
G. Weiss in 2003. The modified differential equation contains a new
nonlinear damping term, that is maximal monotone and possibly
set-valued. We show that this new class of nonlinear infinite
dimensional systems is incrementally scattering passive (hence
well-posed). Our approach uses the theory of maximal monotone
operators and the Crandall-Pazy theorem about nonlinear contraction
semigroups, which we apply to a Lax-Phillips type nonlinear semigroup
that represents the whole system. We illustrate our result on the
$n$-dimensional wave equation.
\end{abstract}


% REQUIRED
%\begin{AMS}
%  68Q25, 68R10, 68U05
%\end{AMS}

%%%%%%%%%%**********%%%%%%%%%%**********%%%%%%%%%%**********%%%%%%%%%%
\section{Introduction} \label{sec1} % Section 1

The dynamics of many physical processes represented by partial
differential equations can be formulated as abstract second order
differential equations in time. While modeling such processes, we
often encounter nonlinear damping terms, for instance, static friction
for a beam equation (see \cite{Berr,Miletic,ShWeTu:art149,Zhijian}),
damping due to the viscosity of the medium for a wave equation (see
\cite{ConRao,Haraux,Todorova}), nonlinear boundary damping for
Maxwell's equations (see \cite{ELN_decay}) and nonlinear conductivity
due to the presence of a semiconductor in the domain for Maxwell's
equations (see \cite{Pokojovy, ShWeTu:21}). The existence of unique
classical and generalized solutions (to be defined later) of such
nonlinear systems is not guaranteed. This paper is about the
well-posedness of a class of second order systems on real Hilbert
spaces in the presence of nonlinear (possibly multi-valued) damping.

Extensive studies investigating systems described by linear partial
differential equations with a nonlinear damping term, acting in the
interior or on the boundary of the domain, have been conducted, for
instance \cite{AlabAmma,Barbu,Chitour,ELN_decay,Haraux,Lasi_1989,LaTa,
Ramirez:17,Todorova,Tri,Zhijian,Zuazua_1990}. As far as we are aware,
most of the papers on this topic treat the well-posedness of the
associated Cauchy problem (without considering inputs and outputs). In
our previous work \cite{ShWeTu:20,ShWeTu:art149} we have made an
attempt to bridge this gap by considering nonlinear infinite
dimensional systems with both input and output signals, wherein the
nonlinear (possibly set-valued) damping term was assumed to be defined
on the entire state space.  Using the theory of maximal monotone
operators and Lax-Phillips semigroups, we have proved that, under
suitable assumptions, such systems are incrementally scattering
passive (hence well-posed). In this article we consider a special
class of systems with a nonlinear (possibly set-valued) damping term
that is only densely defined. This allows to considerably enlarge the
class of damping operators (as compared to \cite{ShWeTu:20,
ShWeTu:art149}), to include also damping via boundary operators, as
well as distributed damping operators that are not defined on all the
state space, and therefore could not be handled with the results in 
\cite{ShWeTu:art149}, see Section \ref{sec7} for examples.

There are some noteworthy recent articles about nonlinear
perturbations of linear infinite dimensional systems with input and
output signals. For instance, in \cite{Hastir_Hans} the well-posedness
of linear systems with nonlinear feedback is discussed. In
\cite{Ramirez:17}, one-dimensional port-Hamiltonian systems with a
nonlinear controller acting on the boundary are investigated (well
posedness and stability). In \cite{Marx}, global asymptotic stability
of linear infinite dimensional systems subject to nonlinear damping
has been studied via the Lyapunov stability criterion, with the
assumption that the origin of the system is globally asymptotically
stable with a linear damping. In these articles the nonlinear damping
term must obey certain continuity conditions, for instance, in
\cite{Hastir_Hans} and \cite{Marx} the nonlinear damping must be
locally Lipschitz. In \cite{SchmidJacob} the well-posedness of a class
of time-varying semilinear systems with boundary control and
observation is proved. In \cite{SchmidJacob}, a time varying
nonlinearity is considered, which is Lipschitz continuous on bounded
subsets of the state space. In the recent survey \cite{Mironchenko},
input to state stability of nonlinear infinite dimensional systems is
discussed. However, the type of nonlinear perturbation that they
consider and their assumptions are rather different from ours.

The aim of this paper is to investigate the well-posedness of a class
of nonlinear infinite dimensional systems. We consider systems on the
real Hilbert spaces $H$ and $U$, where $H$ and $U$ are identified with
their duals. The state trajectories of the systems that we consider
obey the differential inclusion \vspace{-1mm}
\begin{equation} \label{Inclusion1}
  \ddot{z}(t) \m\in\m -A_0z(t)-\half B_0C_0\dot{z}(t)
  -\Nscr(\dot{z}(t))+ B_0u(t) \quad (t\geq 0),\vspace{-1mm}
\end{equation}
where $A_0:\Dscr(A_0)\rarrow H$ is a self-adjoint, positive and
boundedly invertible operator on $H$. We define $H_\half=\Dscr
(A_0^\half)$, with the inner product $\langle x,z\rangle_\half=\langle
A_0^\half x, A_0^\half z\rangle_H$ and the corresponding norm
$\|.\|_\half$. We denote by $H_{-\half}$ the dual of $H_\half$ with
respect to the pivot space $H$, see \cite[Sect. 3.4]{obs_book}. $A_0$
can be extended to a bounded operator from $H_\half$ to $H_{-\half}$,
and we denote this extension by the same symbol. Here $B_0\in \Lscr(U,
H_{-\half})$, that is, $B_0$ is a linear bounded operator from $U$ to
$H_{-\half}$, and $C_0=B_0^*$, so that $C_0\in\Lscr(H_\half, U)$.

In \rfb{Inclusion1} $\Nscr$ is a nonlinear (possibly multi-valued)
densely defined and maximal monotone operator from $H_\half$ to
$H_{-\half}$ with domain $\Dscr(\Nscr)\subset H_\half$. $\Nscr$ being
{\em densely defined} means that $\Dscr(\Nscr)$ is dense in $H_\half$.
We refer to Sect.~\ref{sec2} for some background on maximal monotone
operators. In \rfb{Inclusion1} and later, by $\dot{z}(t)$ and
$\ddot{z}(t)$ we mean the first and second derivative of $z$ with
respect to the time $t$ \m {\em from the right}.

The output of the system under consideration is given by 
\begin{equation} \label{Output1} \vspace{-1mm}
  y(t) \m=\m -C_0\dot{z}(t) + u(t).
\end{equation}
The state of the system is $x(t)=\sbm{z(t) & \dot{z}(t)}^\top$ and
this evolves in the state space $X=H_\half\times H$, which is a
Hilbert space with the usual inner product for product spaces. This
class of systems is a generalization of the ``conservative linear
systems out of thin air" introduced in \cite{TuWe,WeTu03}, the novelty
being the appearance of the nonlinear damping term $\Nscr$.

Our main result can be formulated briefly as follows: The system
\rfb{Inclusion1}, \rfb{Output1} has a unique generalized solution (to
be defined) for any initial state $x(0)\in X$ and any input function
$u\in L^2([0,\infty);U)$. Moreover, the system is {\it incrementally
scattering passive}, i.e., its generalized solutions satisfy the
following energy balance inequality: for any $\tau\geq 0$,
\begin{equation} \label{energy_bal_M}
  \|x_1(\tau)-x_2(\tau)\|_X^2 + \int_0^\tau \|y_1(t)-y_2(t)\|_U^2 
  \dd t \leq\m \|x_{01}-x_{02}\|_X^2 + \int_0^\tau 
  \|u_1(t)-u_2(t)\|_U^2 \dd t ,
\end{equation} 
where $x_{01}=\sbm{z_{_{01}}\\ w_{_{01}}}$, $x_{02}=\sbm{z_{_{02}}\\
w_{_{02}}}$ are the initial states in \m $X=H_{\half}\times H$, $u_1$,
$u_2$ are the input functions in $\LE U$, and we denote by $x_1=
\sbm{z_1\\ \dot{z}_1}$, $x_2=\sbm{z_2\\ \dot{z}_2}$ the corresponding
state trajectories of the system and by $y_1$, $y_2$ the corresponding
output functions. While this statement may seem intuitive (we know
from \cite{WeTu03} that it holds for $\Nscr=0$, and we expect that
extra damping cannot hurt), actually proving it is not simple, as we 
shall see in Sect.~\ref{sec5}.

We briefly recall some background on monotone operators and nonlinear
contraction semigroups in Sect.~\ref{sec2} and on well-posed linear
systems in Sect.~\ref{sec3}. In Sect.~\ref{sec4} we introduce
well-posed nonlinear systems via Lax-Phillips semigroups. In Sect.
\ref{sec5} we introduce classical and generalized solutions and give
our main well-posedness result and its proof. In Sect.~\ref{sec6} we
prove a version of our main result for systems with nonlinear boundary
damping. In Sect.~\ref{sec7} we apply our main result to the wave
equation on an $n$-dimensional bounded domain with two possible
nonlinear damping terms, one distributed and one on the boundary.

%%%%%%%%%%**********%%%%%%%%%%**********%%%%%%%%%%**********%%%%%%%%%%
\section{Monotone operators and nonlinear contraction semigroups}
\label{sec2} % Section 2

In this section we recall some preliminaries about monotone operators
on real Hilbert spaces and some standard results about the nonlinear
operator semigroups defined using such operators. The following
results are based on \cite{Barbu,Brezis,Brezis2,Browder68,CranPazy, 
Kato_accr,Minty,Rocka}. For modern introductions see
\cite{Bauschke,Show}.

Let $Z$ be a real Hilbert space and let $Z'$ be its dual. A nonlinear
(possibly multi-valued) operator $\Nscr$ defined on $\Dscr(\Nscr)
\subset Z$, whose values are nonempty subsets of $Z'$, is {\em
monotone} if for every $w_1, w_2\in\Dscr(\Nscr)$, the inequality
\vspace{-1mm}
\begin{equation} \label{monotone_op}
   \m\hspace{19mm}\langle g_1-g_2, w_1-w_2\rangle \m\geq\m 0 
   \FORALL g_1\in\Nscr(w_1),\m g_2\in\Nscr(w_2),\vspace{-1mm}
\end{equation}
holds. Such an operator is called {\it maximal monotone} if $\Nscr$
has no proper monotone extension (mapping a subset of $Z$ to subsets
of $Z'$). An important characterization of maximal monotone operators
was given by G. Minty in \cite{Minty}: {\color{blue}

\begin{theorem}[Minty]\label{Minty}
Let $Z$ be a real Hilbert space, and we identify $Z'$ with $Z$. A
monotone operator $\Nscr$ with domain $\Dscr(\Nscr)\subset Z$ is
maximally monotone if and only if \vspace{-2mm}
\begin{equation} \label{maxmono}
   {\rm Ran} \m (I+\Nscr) \m=\m Z,\vspace{-1mm}
\end{equation} 
where by the range ${\rm Ran}(I+\Nscr)$ we mean the union of all the
sets $h+\Nscr(h)$, where $h\in\Dscr(\Nscr)$. Moreover if \rfb{maxmono}
holds, then $(I+\Nscr)^{-1}$ is a single-valued contraction operator.
\end{theorem}}

If $\Nscr$ is maximal monotone, then its graph is closed, for every 
$w\in\Dscr(\Nscr)$, the set $\Nscr(w)$ is closed and convex and also
the closure of $\Dscr(\Nscr)$ is a convex set.

Given two maximal monotone operators $\Nscr_1$ and $\Nscr_2$, their
sum $\Nscr_1+\Nscr_2$ is not necessarily maximal monotone as it might
happen that $\Dscr(\Nscr_1)\cap \Dscr(\Nscr_2)$ is too small or even
empty. In this regard, an important theorem to determine the
maximality of the sum of two maximal monotone operators was given by
R.T. Rockafellar, see \cite[Theorem 1]{Rocka}.

{\color{blue}
\begin{theorem}[Rockafellar] \label{Rockafellar} If $\Nscr_1$ and 
$\Nscr_2$ are maximal monotone operators on $Z$ and $\Dscr(\Nscr_1)
=Z$, then $\Nscr_1+\Nscr_2$ is maximal monotone.
\end{theorem}} 

The above theorem is only a particular case of Theorem 1 in
\cite{Rocka}.

Maximal monotonicity is an essential property in the theory of
strongly continuous semigroups of nonlinear contraction operators. The
study of existence and uniqueness of solutions of many evolution
equations (similar to \rfb{Inclusion1}) depends on the generation
theory of such semigroups. We now recall very briefly some facts about
strongly continuous semigroups of nonlinear operators. For the basics
about such semigroups we refer to
\cite{Barbu,Brezis74,CranPazy,Kato_accr,Show,Tanabe}.

Let $Z$ be a real Hilbert space, and we identify $Z'=Z$. A {\em
strongly continuous semigroup of nonlinear operators} $\ULax$ acting
on $Z$ is defined exactly as in the linear case. If $\ULax$ is such a
semigroup, then define the (single-valued) operator \vspace{-2mm}
\begin{equation} \label{operator_minimal}
   \GothA^0 z \m=\m \lim_{t\rarrow 0,\m t>0} \frac{1}{t} \left[ 
   \ULax_t z - z\right], \m\quad \Dscr(\GothA^0) \m=\m
   \left\{ z\in Z\ |\ \mbox{the limit exists} \right\} \m.
\end{equation} 
Following \cite{CranPazy}, $\GothA^0$ is called the {\em (strong)
generator} of $\ULax$. The semigroup $\ULax$ is called {\it
contractive} (or a {\it contraction semigroup}) if \vspace{-1mm}
\[ \|\ULax_t z_1 - \ULax_t z_2\| \m\leq\m \|z_1-z_2\| 
   \n\n\FORALL z_1,z_2\in Z, \m\m t\geq 0.\] 
A set-valued operator $\GothA:\Dscr(\GothA)\rarrow Z$ (with
$\Dscr(\GothA)\subset Z$) is called {\em (maximal) dissipative} if
$-\GothA$ is (maximal) monotone. The following theorem is extracted
from Theorems 1.3, 1.4, 1.5 and A1 as well as Corollary 3.1 in
\cite{CranPazy}, and is formulated to fit the nonlinear operators
discussed in this article. Recall that in this article, $\dot{z}$
denotes the right derivative of $z$.

{\color{blue}
\begin{theorem} \label{Putin}
Assume that $\ULax$ is contractive. Then its generator, $\GothA^0$
from \rfb{operator_minimal} is densely defined and dissipative. This
operator $\GothA^0$ has a unique maximal dissipative extension
$\GothA$ (which might be set-valued) with the same domain
$\Dscr(\GothA)=\Dscr(\GothA^0)$. If $z_0\in \Dscr(\GothA)$, then
$\GothA^0z_0$ is the unique element of smallest norm in the closed and
convex set $\GothA\m z_0$.

Let $z_0\in\Dscr(\GothA)$. The function $z:[0,\infty)\rarrow Z$
defined by $z(t)=\ULax_tz_0$ is Lipschitz continuous and right
differentiable at every $t\geq 0$. Moreover, for every $t\geq 0$, it
holds that $z(t)\in\Dscr(\GothA)$, \vspace{-1.5mm}
\begin{equation} \label{d^+z/dt}
  \dot{z}(t) \m=\m \GothA^0 z(t),\vspace{-1mm}
\end{equation}
and $\GothA^0 z$ is right continuous at $t$. The function
$\|\GothA^0z\|$ is nonincreasing. If $t>0$ is such that $\|\GothA^0 z
\|$ is continuous at $t$, then $z$ is differentiable at $t$.
\end{theorem}}

The operator $\GothA^0$ being densely defined means that $\Dscr
(\GothA^0)$ is dense in $Z$. In the linear case, $\GothA^0=\GothA$.
The following theorem due to M.~Crandall and A.~Pazy 
\cite[Theorem A1]{CranPazy} can be considered as a generalization of
the Lumer-Phillips theorem from linear semigroup theory (see for 
instance \cite{Engel_Nagel} or \cite{obs_book}).

{\color{blue}
\begin{theorem}[Crandall-Pazy] \label{Crandall-Pazy}
Let $\GothA$ be a maximal dissipative set-valued operator on $Z$ with
domain $\Dscr(\GothA)$ dense in $Z$. For each $z_0\in\Dscr(\GothA)$
let $\GothA^0z_0$ denote the element of smallest norm in $\GothA\m
z_0$. Then there is a unique strongly continuous semigroup of
nonlinear operators $\ULax$ acting on $Z$ such that $\GothA^0$ is the
generator of $\ULax$. Moreover, $\ULax$ is contractive.
\end{theorem}}

\begin{remark}
Let $\GothA:\Dscr(\GothA)\rarrow Z$ be a set-valued operator, with
$\Dscr(\GothA)$ dense in $Z$. With an initial condition $z(0)=z_0\in
\Dscr(\GothA)$, we define an {\it abstract nonlinear Cauchy problem}
as follows: find a continuous and right differentiable $z:[0,\infty)
\rarrow Z$ such that $z(t)\in\Dscr(\GothA)$ for all $t\geq 0$ and
\vspace{-3mm}
\[ \dot{z}(t) \m\in\m \GothA\m z(t)\quad \FORALL t\geq 0, 
   \quad z(0)=z_0. \]
From Theorems \ref{Putin} and \ref{Crandall-Pazy}, we obtain that this
abstract Cauchy problem has a unique solution if $\GothA$ is a maximal
dissipative operator. Moreover, $z(t)=\ULax_t z_0$, where $\ULax$ is
the semigroup of contractions generated by $\GothA^0$.
\end{remark}

Denote by $C^1_r(J;Z)$ the space of continuous $Z$-valued functions
$f$ on the interval $J$ such that for every $t\in J$ where $t\neq {\rm
sup}\m\m J$, $f$ is right differentiable and its right derivative is
right continuous and bounded on $J$. Such functions are the integral
of their right derivatives, meaning that for any $a,b\in J$, $a<b$,
$f(b)-f(a)=\int^b_a\dot{f}(t)\dd t$, see for instance \cite{wiki}.

{\color{blue}
\begin{proposition} \label{Lebesgue} 
If $f\in C^1_r(J;Z)$ then $f$ is differentiable at almost every 
$t\in J$.
\end{proposition}}

{\em Proof.} If $f\in C^1_r(J;Z)$, then its right derivative at $t\in
J$ denoted by $\dot{f}(t)$ is right continuous and bounded.
Consequently, $\dot{f}(t)$ is locally Bochner integrable on $J$ which
in turn implies that almost every $t\in J$ is a Lebesgue point (see
\cite[p. 49]{Diestel}). Therefore, at almost every $t\in J$,
\vspace{-1.5mm}
\[ \dot{f}(t) \m=\m \lim_{\e\rarrow 0} \frac{1}{\e} \int^{t+\e}_\e 
   \dot{f}(\sigma)\dd \sigma \m=\m \lim_{\e\rarrow 0} \frac{f(t+\e)
   -f(t)}{\e},\vspace{-2mm} \]
which means that $f$ is differentiable at $t$. $\m\ \square$ 

\smallskip

A single-valued operator $\Nscr_0$ from a real Hilbert space $Z$ to
its dual $Z'$ is called {\it hemi-continuous} if it has the following
property: whenever $u\in\Dscr(\Nscr_0)\subseteq Z$, $(t_n)$ is a
sequence in $(0,\infty)$, $t_n\rarrow 0$, $v\in Z$ and $u+t_n v\in
\Dscr(\Nscr_0)$, then $\Nscr_0(u+t_n v)\rarrow\Nscr_0(u)$ in the
weak topology of $Z'$. For details see \cite{CranPazy} or
\cite[Sect.~6.3]{Tanabe}.

The following proposition follows from \cite[Lemma 2.5]{CranPazy}, see
also \cite[Theorem 2.4]{Barbu}.  

{\color{blue} \begin{proposition}\label{Hemicont}
A hemi-continuous and monotone operator $\Nscr_0:Z\rarrow Z'$, with \\
$\Dscr(\Nscr_0)=Z$, is maximal monotone.
\end{proposition}} \vspace{-2.5mm}

%%%%%%%%%%**********%%%%%%%%%%**********%%%%%%%%%%**********%%%%%%%%%%
\section{Well-posed linear systems} \label{sec3} % Section 3

We recall some background on well--posed linear systems, following
\citep{Sala87,Staf_book,StWe02,TuWe_survey,Weiss10,WeStTu01}. We use
the standard notation from functional analysis, such as $\Lscr(X,Y)$
and $\Dscr(\Nscr)$ introduced in Sect.~\ref{sec1}, and $\rho(A)$ for
the resolvent set of $A$. For any interval $J$, $L^2(J;U)$ denotes
the space of $U$-valued $L^2$ functions defined on $J$, while
$\Hscr^1(J;U)$ is the Sobolev space of functions in $L^2(J;U)$ that
are integrals of functions in $L^2(J;U)$. For $n\in\{0,1,2,...\}$, $\m
C^n(J;U)$ is the space of the $n$ times continuously differentiable
functions on $J$, while $BC^n(J;U)$ is the space of bounded functions
in $C^n(J;U)$. If $n=0$, we omit to write it. $C^1_r(J;U)$ has been
introduced in Sect.~\ref{sec2}.

Let us denote by $U$ the {\it input space}, by $X$ the {\it state
space} and by $Y$ the {\it output space} of a well-posed linear system
\m $\Sigma^{\rm L}$ (these are Hilbert spaces). The input and the
output functions are $u\in\LE U$ and $y\in\LEloc Y$, respectively. For
any $y\in\LEloc U$ and any $\tau\geq 0$, we denote by $\PPP_\tau y$ 
the truncation of $y$ to the interval $[0,\tau]$. According to the 
standard meaning of the notation $L^2_\loc$, $\PPP_\tau y$ is in $L^2
([0,\infty);Y)$ and it is zero for $t>\tau$.

{\it A well-posed linear system} $\Sigma^{\rm L}$ consists of a family
of bounded linear operators $\Sigma^{\rm L}=(\Sigma^{\rm
L}_\tau)_{\tau\geq 0}$ such that \vspace{-4mm}
\begin{equation} \label{UNUZERO}
  \bbm{ x(\tau) \\ \PPP_\tau y } \m=\m \Sigma^{\rm L}_\tau
  \bbm{ x(   0) \\ \PPP_\tau u } \m.\vspace{-2mm}
\end{equation}
Here $x:[0,\infty)\rarrow X$ is the {\em state trajectory} of
$\Sigma^{\rm L}$ corresponding to the initial state $x(0)$ and the
input function $u$, and $y$ is the corresponding output
function. Denoting $c_\tau=\|\Sigma^{\rm L}_\tau\|$, we have
\vspace{-2mm}
\begin{equation} \label{wellposed_ineq}
\| x(\tau)\|^2 + \int_0^\tau \left\| y(t)\right\|^2 \dd t
   \m\leq\m c_\tau^2 \left( \| x(   0)\|^2 + \int_0^\tau \left\| 
   u(t)\right\|^2 \dd t \right). \vspace{-2mm}
\end{equation} 

The operators $\Sigma^{\rm L}_\tau$ are partitioned in a natural way:
\vspace{-2mm}
\begin{equation} \label{Sig4b}
   \Sigma^{\rm L}_\tau \m=\m \bbm{\tline_\tau & \Phi_\tau\\
   \Psi_\tau & \fline_\tau} \m.\vspace{-2mm}
\end{equation}
The four families of operators appearing on the right-hand side above
must satisfy four functional equations expressing the causality and
the time-invariance of \m $\Sigma^{\rm L}$ (these functional equations
are parts of the definition of a well-posed system), see for instance
\cite{Weiss10}. In particular, the family $(\tline_\tau)_{\tau\geq 0}$
is a strongly continuous operator semigroup on $X$ and its generator
$A$ is called the {\em semigroup generator} of $\Sigma^{\rm L}$. We
introduce $X_1=\Dscr(A)$ with the norm $\|x\|_1=\|(\beta I-A)x\|$,
where $\beta\in \rho(A)$. $X_{-1}$ is the completion of $X$ with
respect to the norm $\|x\|_{-1}=\|(\beta I-A)^{-1}x\|$. These spaces
are independent of the choice of $\beta$. $A$ has a unique extension
that is bounded from $X$ to $X_{-1}$, and we denote this extension by
the same symbol $A$. The semigroup $\tline$ can be extended to an
operator semigroup on $X_{-1}$, denoted by the same symbol, whose
generator is the extension of $A$ mentioned earlier. There exists a
unique operator $B\in \Lscr(U;X_{-1})$, called the {\em control
operator} of $\Sigma^{\rm L}$, such that for all $t\geq 0$,
\vspace{-1mm}
\[ \Phi_t u \m = \m \int^t_0 \tline_{t-\sigma} Bu(\sigma)\dd \sigma 
   \quad\quad \forall u\in L^2([0,\infty);U). \vspace{-1mm}\]
The above integration is done in $X_{-1}$. There exists a unique {\em 
observation operator} $C\in \Lscr(X_1,Y)$ so that for every
$\tau\geq 0$,\vspace{-1mm}
\[ (\Psi_\tau x_0)(t) \m=\m C\tline_t x_0 \quad\quad \forall x_0\in
   \Dscr(A),\ t\in[0,\tau]. \vspace{-1mm}\]

The operator $C$ has an extension $\bar{C}$ to the space $Z$ defined
as:\vspace{-1mm}
\begin{equation} \label{SpaceZ}
   Z \m=\m \Dscr(A)+(\beta I-A)^{-1}BU. \vspace{-1mm}
\end{equation}
This is a Hilbert space with the norm \vspace{-1mm}
\begin{equation} \label{Chauvin_convicted}
   \|z\|^2_Z \m=\m \inf\left\{\ \nm\| x \|^2_1+\| v \|^2
   \left|\begin{array}{c} x \in X_1,\ \ v \in U,\\ z=x+(\beta I-A
   )^{-1}Bv \end{array}\right.\right\},\vspace{-1mm}
\end{equation}
and $\bar{C}\in\Lscr(Z,Y)$. The extension $\bar{C}$ may not be unique.
For each such extension $\bar{C}$, there exists $D\in\Lscr(U,Y)$ such
that the transfer function $\GGG$ of $\Sigma^{\rm L}$ is given by
$\GGG(s)=\bar{C}(sI-A)^{-1}B+D$, for all $s\in \cline$ with ${\rm
Re}\m s>\omega_0$, where $\omega_0$ denotes the growth bound of
$\tline$.  The following proposition is contained in 
\citep[Theorem~3.1]{StWe02}.

{\color{blue}
\begin{proposition} \label{YomKippur}
We use the notation introduced earlier in this section. Assume that $u
\in\Hscr^1((0,\infty);U)$ and $x_0\in X$ are such that
$Ax_0+Bu(0)\in X$.  The state trajectory $x$ and the output function
$y$ of $\m\m\Sigma^{\rm L}$ are defined as in \rfb{UNUZERO}. Then,
denoting $e_\alpha(t)=e^{\alpha t}$ (for $\alpha\in\rline,\,t\geq 0$),
we have \vspace{-5mm}
\[ x\in C^1([0,\infty);X), \quad Ax+Bu\in C([0,\infty);X),\quad 
   e_{-\o} y \in \Hscr^1((0,\infty);Y),\vspace{-1mm} \]
for all $\o>\max\{0,\o_0\}$, and for every $t\geq 0$ we have that
\vspace{-2mm}
\begin{equation} \label{Kurds_abandoned}
  \frac{\dd x(t)}{\dd t} \m=\m Ax(t) + Bu(t), \quad y(t) \m=\m 
  \bar{C}x(t)+Du(t).
\end{equation}
\end{proposition}}

Now we recall the class of conservative linear systems ``out of thin air" that were
already mentioned in Sect.~\ref{sec1}. Using the notation $U$, $H$,
$H_\half$, $H_{-\half}$, $A_0$, $B_0$, $C_0$ as in Sect.~\ref{sec1},
we also introduce $H_1=\Dscr(A_0)$, with the inner product $\langle
x,z \rangle_{H_1}=\langle A_0x, A_0z \rangle$ and the corresponding
norm $\|.\|_{H_1}$.
Consider the linear system $\Sigma^{\rm L}$ described by: \vspace{-1mm}
\begin{equation} \label{LinearSys}
  \frac{\dd^2 z(t)}{\dd t^2}+A_0z(t) + \half B_0\frac{\dd}{\dd t} C_0
  z(t) = B_0 u(t), \quad y(t) \m=\m -C_0\frac{\dd z(t)}{\dd t} + u(t).
\end{equation}
We know from Theorems 1.1 and 1.2 of \cite{WeTu03} that the equations
\rfb{LinearSys} determine a well-posed system with state $x(t)=\sbm
{z(t)\\ \frac{\dd z(t)}{\dd t}}$ and state space $X=H_\half\times H$. Moreover, if
$u\in\Hscr^1((0,\infty);U)$, and the initial conditions satisfy $z(0)=
z_0\in H_\half$, $\frac{\dd z}{\dd t}(0)=w_0\in H_\half$ and
\vspace{-3mm}
\begin{equation} \label{BCondition}
  A_0 z_0 + \half B_0 C_0 w_0 - B_0 u(0) \m\in\m H \m,
\end{equation}
then the solution $z$ and the output $y$ from \rfb{LinearSys} satisfy 
\vspace{-2mm}
\begin{equation} \label{Solution_z}
  z\in BC([0,\infty);Z_0) \cap BC^1([0,\infty); H_{\half})\cap 
  BC^2([0,\infty); H),\m\quad y\in\Hscr^1((0,\infty);U),\vspace{-1mm}
\end{equation}
where $Z_0=\Dscr(A_0)+A_0^{-1}B_0U\subset H_\half$, with the norm 
\vspace{-2mm}
\begin{equation} \label{Matty}
   \|z\|^2_{Z_0} \m=\m {\rm inf} \left\{\|z_1\|_{H_1}^2+\|v\|^2
   \left|\ \ z=z_1+A_0^{-1}B_0v,\ \ z_1\in\Dscr(A_0),
   \ \ v\in U \right.\right\}.
\end{equation}
Furthermore, these solutions satisfy \rfb{wellposed_ineq} with 
equality and with $c_\tau=1$: \vspace{-2mm}
\[ \left\| \bbm{z(\tau)\vspace{1mm}\\ \left(\frac{\dd z}{\dd t}\right)
   (\tau)}\right\|^2 + \int_0^\tau \left\| y(t)\right\|^2 \dd t \m=\m 
   \left\| \bbm{z(0)\vspace{1mm}\\ \left(\frac{\dd z}{\dd t}\right)
   (0)} \right\|^2 + \int_0^\tau \left\| u(t)\right\|^2 \dd t .\]
The equations \rfb{LinearSys} are equivalent to \rfb{Kurds_abandoned} 
with $D=I$ and with $A$, $B$, $\bar{C}$ defined by \vspace{-2mm}
\begin{equation} \label{equi_lin0}
   A \m=\m \bbm{0 & I\\ -A_0 & -\half B_0C_0}, \quad\m\m\m\m B \m=\m
   \bbm{0\\ B_0}, \vspace{-2mm}
\end{equation}
\[ \Dscr(A) \m=\m \left\{\bbm{z\\w}\in H_\half\times H_\half \left|
   \m\m\m\m A_0z+\half B_0 C_0w\in H\right.\right\},\vspace{-2mm}\]
\begin{equation} \label{equi_lin1}
   \bar{C}:Z_0\times H_\half\rarrow U, \m\m\m\quad \bar{C} \m=\m
   \bbm{0 & -C_0}. \vspace{-1mm}
\end{equation}
We refer to \cite{StWe12, WeTu03} for details and for the
proofs of the results mentioned  above. \vspace{-0.1mm}

%%%%%%%%%%**********%%%%%%%%%%**********%%%%%%%%%%**********%%%%%%%%%%
\section{Lax-Phillips semigroups and well-posed nonlinear systems}
   \label{sec4} % Section 4

Following \cite{StWe02} we can define a well-posed linear system
$\Sigma^{\rm L}$ by a semigroup $\ULax^{\rm L}$ called its {\em
Lax-Phillips semigroup}. This semigroup contains the entire
information about the system.

For any $\tau\geq 0$, we denote by $\SSS_\tau$ the (unilateral) right
shift operator by $\tau$ on $\Uscr=L^2([0,\infty);U)$ and also on
$\Yscr=L^2((-\infty,0];Y)$, so that their adjoints $\SSS^*_\tau$ are
the operators of left shift by $\tau$ on the same spaces. We also
introduce $\Sscr_t$, the bilateral right shift by $t$ acting on
$L^2((-\infty,\infty);Y)$ (where $t\in\rline$). We regard $\Yscr$ as a
subspace of $L^2((-\infty,\infty);Y)$, by extending 
%\smallskip \eject \noindent 
functions in $\Yscr$ to be zero for $t>0$. We regard $L^2([0,t];Y)$ as
a subspace of $L^2([0,\infty);Y)$ (by extending functions to be zero
outside $[0,t]$). Recall the notation $\PPP_\tau$ from
Sect. \ref{sec3}.

{\color{blue}
\begin{proposition} \label{LaxPhilProp}
Let $\Sigma^{\rm L}$ be a well-posed linear system as in \rfb{UNUZERO}
and \rfb{Sig4b}. For all $t\geq 0$, we define on $\Yscr\times
X\times\Uscr$ the operator $\ULax^{\rm L}_t$ by \vspace{-1mm}
\[ \ULax^{\rm L}_t \m=\m \bbm{\Sscr_{-t} & 0    & 0 \\
                      0      & I            & 0 \\
                      0      & 0            & \SSS_t^*}
             \bbm{    I      & \Psi_t       & \fline_t\\
                      0      & \tline_t     & \Phi_t\\
                      0      & 0            & I}.\vspace{-2mm}\]
Then $\ULax^{\rm L}=(\ULax^{\rm L}_t)_{t\geq 0}$ is a strongly 
continuous semigroup of linear operators.
\end{proposition}}

If we take $y_0\in\Yscr$, $x_0\in X$ and $u_0\in\Uscr$ to represent
the past output function of $\Sigma^{\rm L}$ (for $t<0$), its initial
state and its input function respectively, then at any time $t\geq 0$,
the first component of $\ULax^{\rm L}_t\sbm{y_0 \\ x_0 \\ u_0}$
represents the past output up to $t$, the second component represents
the present state $x(t)$ and the third component represents the future
input that will reach $\Sigma^{\rm L}$ after $t$. The generator of
$\ULax^{\rm L}$ can be characterized as follows: 

{\color{blue}
\begin{proposition} \label{LaxPhillipsGenTh}
Let $\ULax^{\rm L}$ be the Lax--Phillips semigroup of the well-posed
linear system $\Sigma^{\rm L}$. The generator of $\ULax^{\rm L}$ is
denoted by $\GothA^{\rm L}$. Using the notation $A$, $B$, $\bar{C}$,
$D$ from \rfb{Kurds_abandoned}, the domain of $\GothA^{\rm L}$ is
\vspace{-4mm}
\[ \Dscr(\GothA^{\rm L}) \m=\m \left\{\bbm{y_0\vspace{2mm}\\ x_0
   \vspace{2mm}\\ u_0}\in\left.\begin{array}{c c c} \Hscr^1
   ((-\infty,0);Y)\\ \times\\ X\\ \times\\ \Hscr^1((0,\infty);U)
   \end{array}\right| \begin{array}{c c c} Ax_0+Bu_0(0)\in X,
   \vspace{2mm}\\ y_0(0)=\bar{C}x_0 + Du_0(0)\end{array}\right\}.
   \vspace{-1mm} \]
The generator $\GothA^{\rm L}$ is given by \vspace{-2mm}
\begin{equation} \label{GothAact}
   \GothA^{\rm L} \bbm{y_0 \\ x_0 \\ u_0} \m=\m
   \bbm{ \vspace{1mm} y_0' \\ \vspace{1mm} Ax_0+Bu_0(0) \\ 
   u_0'}\quad \FORALL \bbm{y_0\\ x_0\\ u_0}\in\Dscr(\GothA^{\rm L}).
   \vspace{-1mm}
\end{equation}
\end{proposition}}

We denote by $\delta^*_0$ the operator of point evaluation at $0$,
regarded as an operator from $\Hscr^1((0,\infty);U)$ to $U$, and
$\delta_0$ is its adjoint, mapping $U$ into $\Hscr^1((0,\infty);U)'$
(the dual space of $\Hscr^1((0,\infty);U)$ with respect to the pivot
space $\Uscr$).

\begin{remark} \label{GothA_equivalent} \rm
The operator $\GothA^{\rm L}$ from \rfb{GothAact} can also be written
in the form \vspace{-2mm}
\begin{equation} \label{GothA_LP}
   \GothA^{\rm L} \m=\m \bbm{\left[\frac{\dd}{\dd \xi}\right]_\Yscr & 
   \delta_0\bar{C} & \delta_0 D\delta_0^*\\ 0 & A & B\delta_0^*\\ 0 
   & 0 & \left[\frac{\dd}{\dd \xi}\right]_\Uscr}, \vspace{-2mm}
\end{equation} 
where $\left[\frac{\dd}{\dd\xi}\right]_\Yscr$ is the generator of the
left shift semigroup on $\Yscr$, whose domain is $\Yscr_1=\Hscr^1_0
((-\infty,0);Y)$, and $\left[\frac{\dd}{\dd\xi}\right]_\Uscr$ is the
generator of the left shift semigroup on $\Uscr$, whose domain is
$\Uscr_1=\Hscr^1((0,\infty);U)$. Thus, for any $\varphi\in\Hscr^1
((-\infty,0);Y)$, $\left[\frac{\dd}{\dd\xi}\right]_\Yscr\varphi$ is in
$\Yscr_{-1}$, the dual space of $\Hscr^1((-\infty,0);Y)$ with respect
to the pivot space $\Yscr$, and \vspace{-1mm}
\begin{equation} \label{Kirsten_Morris}
  \left[\frac{\dd}{ \dd \xi}\right]_\Yscr \varphi \m=\m \varphi' - 
  \delta_0\varphi(0) \m,
\end{equation}
where $\varphi'$ is the usual derivative of $\varphi$, see for
instance Example 4.2.7 in \cite{obs_book}.
\end{remark}

Following our article \cite{ShWeTu:art149}, we define a (possibly
nonlinear) well-posed system via its (nonlinear version of the)
Lax-Phillips semigroup.

\begin{definition} \label{abstractnonlinear} 
A time invariant well-posed (possibly nonlinear) system $\Sigma^\NL$
with input space $U$, state space $X$ and output space $Y$ consists of
two families of (possibly nonlinear) continuous operators
\vspace{-2mm}
\[ \Sigma^{\rm  st} \m=\m (\Sigma^{\rm  st}_t)_{t\geq 0} \m,\quad
   \Sigma^{\rm out} \m=\m (\Sigma^{\rm out}_t)_{t\geq 0} \m, \]
where $\Sigma^{\rm st}_t:X\times\Uscr\rarrow X$ and $\Sigma
^{\rm out}_t:X\times\Uscr\rarrow L^2([0,t];Y)$ such that the 
following is a strongly continuous semigroup of (possibly nonlinear)
operators $\ULax^\NL=(\ULax^\NL_t)_{t\geq 0}$ on $\Yscr
\times X\times\Uscr${\rm :} \vspace{-1mm}
\begin{equation} \label{J_Pasha}
   \ULax^\NL_t \m=\m \bbm{\Sscr_{-t} & 0 \sbluff    & 0 \\
   0 & I & 0 \\ 0 & 0 & \SSS_t^*} \left[\begin{array}{c|c c} I & \ 
   \Sigma^{\rm out}_t \\ 0 & \ \Sigma^{\rm st} _t \\ \hline
   \ \ 0 \ \  & 0\ \ \ I \end{array}\right]\quad \FORALL t\geq 0.
\end{equation}
Moreover, we require that for all $\tau\geq 0$, \vspace{-1mm}
\begin{equation} \label{Identities}
   \Sigma^{\rm st}_\tau\bbm{x_0\\ v} =\m \Sigma^{\rm st}_\tau
   \bbm{x_0\\ \PPP_\tau v},\ \quad \ \Sigma^{\rm out}_\tau\bbm{x_0\\
   v} =\m \Sigma^{\rm out}_\tau\bbm{x_0\\ \PPP_\tau v},
\end{equation}
for all $v\in\Uscr$ and $x_0\in X$.
\end{definition} 

The identities \rfb{Identities} are called the {\em causality
conditions}. If $x_0\in X$ and $u\in\Uscr$, the function $t\rarrow
x(t)$ defined by $x(t)=\Sigma^{\rm st}_t\sbm{x_0\\ u}$ is called the
{\em state trajectory} of $\Sigma^\NL$ corresponding to the initial
state $x_0$ and the input function $u$. The function $y_\tau=\Sigma
^{\rm out}_\tau\sbm{x_0\\ u}$ is called the {\em output function} of
$\Sigma^\NL$ corresponding to $x_0$ and $u$, restricted to the 
interval $[0,\tau]$. (We can define a locally $L^2$ function $y$ on
$[0,\infty)$ such that for any $\tau\geq 0$, the restriction of $y$
to $[0,\tau]$ is $y_\tau$ introduced earlier.) We refer to 
\cite{ShWeTu:art149} for comments and examples concerning these
concepts.

A time-invariant well-posed (possibly nonlinear) system $\Sigma^\NL$
is called {\em incrementally scattering passive} if its Lax-Phillips
semigroup $\ULax^\NL$ is contractive, or equivalently, if the estimate
\rfb{energy_bal_M} holds. Incrementally scattering passive systems
come with the advantage that they can be described locally in time,
via the generator of $\ULax^\NL$, using Theorem \ref{Putin}. This is a
topic that we shall explore in another paper.

%%%%%%%%%%**********%%%%%%%%%%**********%%%%%%%%%%**********%%%%%%%%%%
\section{The main result} \label{sec5} % Section 5

Following our article \cite{ShWeTu:art149}, we define the classical
and the generalized solutions of \rfb{Inclusion1} and
\rfb{Output1}. We continue to use the notation of Sect. \ref{sec1},
\ref{sec3} and \ref{sec4}. Remember that in this paper, $\dot{z}$ and
$\ddot{z}$ denote the right derivatives of $z$. In this section $Y=U$,
so that $\Yscr=L^2((-\infty,0];U)$.

\begin{definition} \label{Clas_Sol}
Assume that $A_0$ is a positive and boundedly invertible operator on
$H$ with domain $\Dscr(A_0)\subset H$, $\Nscr$ is a maximal monotone
(possibly set-valued) operator from $H_{\half}$ to
$H_{-\half}$ with $\Dscr(\Nscr)$ dense in $H_\half$. Let $B_0\in\Lscr(U;H_{-\half})$, $C_0=B_0^*$ and
$X=H_\half\times H$. Let $u,y\in \Uscr$, $z\in
C([0,\infty); H_\half)\cap C^1([0,\infty);H)$ and denote
$x=\sbm{z\\\dot{z}}$ $(\text{so that}\m\m x\in C([0,\infty);X))$.
 
The triple $(x,u,y)$ is called a {\em classical solution} of
\rfb{Inclusion1} and \rfb{Output1} if
\begin{enumerate}[\rm (a)]
\item $x$ is right differentiable (in $X$) and $\dot{z}(t)\in\Dscr
      (\Nscr)$, for all $t\geq 0$,
\item $u,y\in\Hscr^1((0,\infty);U)$, 
\item $\left(A_0z(t)+\half B_0C_0\dot{z}(t)+\Nscr(\dot{z}(t))-B_0
      u(t)\right)\cap H\neq \emptyset\m$ for all $t\geq 0$,
\item the inclusion \rfb{Inclusion1} and the formula \rfb{Output1} 
      hold for all $t\geq 0$.\\
The triple $(x,u,y)$ is called a {\em generalized solution} of
\rfb{Inclusion1} and \rfb{Output1} if there exists a sequence
$(x_n,u_n,y_n)$ of classical solutions of \rfb{Inclusion1} and
\rfb{Output1}, such that $x_n(t)\rarrow x(t)$ in $X$ for all $t\geq
0$, $u_n\rarrow u$ in $\Uscr$ and $y_n\rarrow y$ in $\Uscr$.
\end{enumerate}
\end{definition}

\begin{remark}\label{Solutions}
For a classical solution $(x,u,y)$ of \rfb{Inclusion1} and
\rfb{Output1}, where $x=\sbm{z\\ \dot{z}}$, we call $(z(0),\dot{z}(0),
u)$ the initial data of this solution. If the system \rfb{Inclusion1}
and \rfb{Output1} has classical solutions for a dense subspace of
initial data $(z_0, w_0, u_0)\in H_\half\times H\times \Uscr$, and
these solutions satisfy \rfb{energy_bal_M}, then it follows that
\rfb{Inclusion1} and \rfb{Output1} determine an incrementally
scattering passive well-posed system $\Sigma^\Nscr$. Indeed, for any
$(z_0,w_0,u_0)\in H_{\half}\times H\times \Uscr$ there exists a
sequence $(z_{n},w_{n},u_{n})\in H_\half\times\Dscr(\Nscr)\times\Hscr
^1(0,\infty);U)$ such that $z_{n}\rarrow z_0$ in $H_\half$, $w_{n}
\rarrow w_0$ in $H$ and $u_{n}\rarrow u_0$ in $\Uscr$. Let $\left(
\sbm{\tilde{z}_{n}\\ \tilde{w}_{n}},u_{n},y_n\right)$ be the classical
solution of \rfb{Inclusion1} and \rfb{Output1} corresponding to the
initial data $(z_n,w_n,u_n)$. According to inequality
\rfb{energy_bal_M}, $\left(\sbm{\tilde{z}_n(t)\\ \tilde{w}_n(t)},
u_n,y_n\right)$ is a Cauchy sequence in $H_\half\times H\times\Uscr
\times\Uscr$. Hence, there exist limits $\left(\sbm{\tilde{z}(t)\\
\tilde{w}(t)},\tilde{u},\tilde{y}\right)$ such that \vspace{-2mm}
\[ \left(\sbm{\tilde{z}_n(t) \\ \tilde{w}_n(t)}, u_{n}, y_n\right)
   \rarrow \left(\sbm{\tilde{z}(t)\\ \tilde{w}(t)}, \tilde{u}, 
   \tilde{y}\right)\in H_\half \times H\times \Uscr\times \Uscr,
   \vspace{1mm} \]
because $H_\half\times H\times\Uscr\times\Uscr$ is complete. Clearly
$\tilde{u}=u_0$. According to Definition \ref{Clas_Sol} $(\tilde{x},
\tilde{u}, \tilde{y})$ is a unique generalized solution of
\rfb{Inclusion1} and \rfb{Output1}, where $\tilde{x}=\sbm{\tilde{z}\\
\tilde{w}}$.
\end{remark}

The following theorem gives sufficient conditions for the differential
inclusion \rfb{Inclusion1} and the equation \rfb{Output1} to have
classical solutions and shows that if $(x,u,y)$ is such a solution,
then it has some additional properties. Most importantly, we conclude
that \rfb{Inclusion1} and \rfb{Output1} determine an incrementally
scattering passive system. A shorter version of this theorem, with
only an outline of the proof, is in our conference paper
\cite{SW:CDC2022}.

{\color{blue}
\begin{theorem}\label{Main}
The system $\Sigma^\Nscr$ defined by the differential inclusion
\rfb{Inclusion1} and output equation \rfb{Output1} is well-posed with
the state space $X=H_{\half}\times H$ and the input and output space
$U$, in the following sense:

If $u\in\Hscr^1((0,\infty);U)$ and the initial state $\sbm{z_{_0}\\
w_{_0}}\in H_\half\times\Dscr(\Nscr)$ satisfy\vspace{-2mm}
\begin{equation} \label{Haaland}
   \left(A_0 z_{_0} + \half B_0 C_0 w_{_0} + \Nscr(w_{_0}) - B_0 
   u(0) \right)\cap H \m\neq\m \emptyset,\vspace{-2mm} 
\end{equation}
then \rfb{Inclusion1} and \rfb{Output1} have a unique classical
solution $(x,u,y)$ such that $x(0)=\sbm{z(0)\\ \dot z(0)}=
\sbm{z_{_0}\\ w_{_0}}$. 

The solution satisfies \vspace{-3mm}
\begin{equation} \label{Well-posed}
  \m\ \ \ \ z\in C^1_r([0,\infty);H_\half),\hspace{5mm} \dot{z}\in 
  C^1_r([0,\infty);H),\vspace{-1mm}
\end{equation}
in particular, $\ddot z(t)\in H$ for all $t\geq 0$. The function $z$
is differentiable in $H_\half$ for almost every $t\geq 0$ and it is
Lipschitz continuous in $H_\half$. The function $\dot{z}$ is
differentiable in $H$ for almost every $t\geq 0$ and it is Lipschitz
continuous in $H$. Moreover, the classical solutions with state
trajectories $x_1=\sbm{z_{_1}\\ \dot{z}_{_1}}$ and \m $x_2=\sbm{z_{_2}
\\ \dot{z}_{_2}}$ emanating from the initial conditions $x_{_{01}}=
\sbm{z_{_{01}}\\ w_{_{01}}}$, \m $x_{_{02}}=\sbm{z_{_{02}}\\ 
w_{_{02}}}\in H_{\half}\times\Dscr(\Nscr)$ with their corresponding
inputs $u_1,u_2\in\Hscr^1((0,\infty);U)$ and outputs $y_1,y_2\in
\Hscr^1((0,\infty);U)$, satisfy the energy balance inequality
\rfb{energy_bal_M}. 

The triples $(z_0,w_0,u)\in H_\half\times\Dscr(\Nscr)\times\Hscr^1
(0,\infty);U)$ that satisfy the compatibility condition \rfb{Haaland}
are dense in $H_\half\times H\times\Uscr$. Thus, by Remark {\rm 
\ref{Solutions}}, the system $\Sigma^\Nscr$ is an incrementally 
scattering passive well-posed system.
\end{theorem}}

{\em Proof.} \m We prove that $\Sigma^\Nscr$ corresponds to a
Lax-Phillips semigroup $\ULax^\Nscr$ acting on $\Yscr\times H_\half
\times H\times \Uscr$, determined by the maximal dissipative operator
\vspace{-1mm}
\begin{equation} \label{Generator}
  \GothA^\Nscr \m=\m \bbm{\left[\frac{\dd}{\dd\xi}\right]_\Yscr & 0 &
  -\delta_0C_0 & \delta_0I\delta^*_0 \\ 0 &  0 & I & 0 \\ 0 & -A_0 & 
  -\half B_0C_0-\Nscr & B_0\delta^*_0\\ 0 & 0 &  0 &\left[\frac{\dd}
  {\dd\xi}\right]_\Uscr}, \vspace{-1mm}   
\end{equation}
with the dense domain 
\begin{equation} \label{dom_nonli_A}
   \Dscr(\GothA^\Nscr) \m=\m \left\{\bbm{y\vspace{2mm}\\ z
   \vspace{2mm}\\ w\vspace{2mm}\\ u}\in \hspace{-2mm}\left.
   \begin{array}{c c c} \Hscr^1((-\infty,0);U)\vspace{-0.5mm}\\ 
   \times\vspace{-0.5mm}\\ H_{\half}\vspace{-0.5mm}\\ 
   \times\vspace{-0.5mm} \\ \Dscr(\Nscr)\vspace{-0.5mm} \\ \times 
   \vspace{-0.5mm}\\ \Hscr^1((0,\infty);U) \end{array}\right| \m\m\m 
   \begin{array}{c c c} \left(A_0z+\half B_0C_0w+\Nscr(w)-B_0u(0)
   \right)\cap H\neq \emptyset, \\  \\y(0)=-C_0w+u(0)\end{array} 
   \hspace{-1mm}\right\}.
\end{equation}

There is an aspect of $\GothA^\Nscr$ that the formula \rfb{Generator}
does not capture: when computing the third entry of
$\GothA^\Nscr\sbm{y & z & w & u}^\top$ according to \rfb{Generator},
we obtain a set in $H_{-\half}$, but we take its intersection with $H$
(which is non-empty by the definition of $\Dscr(\GothA^\Nscr)$). Thus,
$\GothA^\Nscr\sbm{y & z & w & u}^\top $ is a non-empty subset of
$\Yscr\times H_{\half}\times H \times \Uscr$, as it should be. Notice
that if $\Nscr=0$, then the above description of $\GothA^\Nscr$
reduces to \rfb{GothA_LP} if we take $A$, $B$, $\bar{C}$ as in
\rfb{equi_lin0}, \rfb{equi_lin1} and $D=I$.

First we show that $\GothA^\Nscr$ is dissipative. Consider $\sbm{y_1 &
z_1 & w_1 & u_1}^\top$, $\sbm{y_2 & z_2 & w_2 & u_2}^\top\in
\Dscr(\GothA^\Nscr)$, choose $\eta_1\in\Nscr(w_1)$,
$\eta_2\in\Nscr(w_2)$ such that $A_0 z_1+\half B_0 C_0 w_1+\eta_1-B_0
u_1(0)\in H$ and similarly for $\eta_2$. Only such $\eta_1$, $\eta_2$
are legitimate values according to the definition of 
$\Dscr(\GothA^\Nscr)$ and the clarification written after its
formula. Then we evaluate the following inner product on $\Yscr\times
H_{\half}\times H\times \Uscr$:
\[ \hspace{-7cm}\left\langle\GothA^\Nscr\sbm{y_1\\ z_1\\ w_1\\ u_1} 
   -\m\m \GothA^\Nscr\sbm{y_2\\ z_2\\ w_2\\ u_2}, \sbm{y_1\\ z_1\\ 
   w_1\\ u_1}-\sbm{y_2\\ z_2\\ w_2\\ u_2}\right\rangle 
   \vspace{-2mm} \]
\begin{equation} \label{innerprod}
   =\m \left\langle \sbm{\left[\frac{\dd}{\dd\xi}\right]_\Yscr y_1-
   \left[\frac{\dd}{\dd\xi}\right]_\Yscr y_2-\delta_0 C_0 w_1 + 
   \delta_0 C_0 w_2\\ w_1-w_2\\-A_0 z_1+A_0 z_2-\half B_0 C_0 w_1 +
   \half B_0 C_0 w_2\\ \left[\frac{\dd}{\dd\xi}\right]_\Uscr u_1-
   \left[\frac{\dd}{\dd\xi}\right]_\Uscr u_2} + \sbm{\delta_0 u_1(0)
   -\delta_0u_2(0)\vspace{2mm}\\ 0\vspace{2mm}\\ -\eta_1+\eta_2+B_0
   u_1(0)-B_0u_2(0)\vspace{2mm}\\0},\sbm{\vspace{1mm} y_1-y_2 
   \vspace{1mm}\\ z_1-z_2 \vspace{2mm}\\ w_1-w_2 \vspace{2mm}\\ 
   u_1-u_2}\right\rangle.
\end{equation}
According to the definition of the inner product on the space
$H_\half$ and the duality pairing between $H_{-\half}$ and $H_\half$, denoted by $\langle \cdot, \cdot\rangle_{-\half,\half}$, the following holds: \vspace{-2mm}
\[ \hspace{3mm}\langle w_1-w_2, z_1-z_2\rangle_{\half} 
   \m=\m \langle A_0^\half(w_1-w_2), A_0^\half(z_1-z_2)\rangle_H,
   \vspace{-3mm} \]
\[ \hspace{-7mm}\langle A_0(z_1-z_2),w_1-w_2\rangle_{-\half,\half} \m=\m \langle A_0^\half(w_1-w_2), A_0^\half(z_1-z_2)
   \rangle_H \m. \]
Using integration by parts we obtain that \vspace{-2mm} 
\[ \left\langle\left[\frac{\dd}{\dd\xi}\right]_\Uscr(u_1-u_2),\m 
   u_1-u_2\right\rangle_\Uscr=-\half\|u_1(0)-u_2(0)\|^2. \]
Moreover, using \rfb{Kirsten_Morris} and the last condition in the
definition of \m $\Dscr(\GothA^\Nscr)$, we get
\begin{equation}\label{derivativeY} 
\left[\frac{\dd}{\dd \xi}\right]_\Yscr y_1-\delta_0 C_0 w_1 +
   \delta_0 u_1(0) \m=\m y_1', \quad \left[\frac{\dd}{\dd \xi}
   \right]_\Yscr y_2-\delta_0 C_0 w_2+\delta_0 u_2(0) \m=\m y_2'\m.
\end{equation}
%According to \rfb{Kirsten_Morris},
%\[\left[\frac{\dd}{\dd \xi}\right]_\Yscr y_1=y_1'-\delta_0\left(-C_0w_1+u_1(0)\right)\m,\m\m\m\m \left[\frac{\dd}{\dd \xi}\right]_\Yscr y_2 = y_2'-\delta_0\left(-C_0w_2+u_2(0)\right)\m. \]
Using these last five formulas we get that \rfb{innerprod} implies
\[ \hspace{-8mm}\left\langle\GothA^\Nscr\sbm{y_1\\ z_1\\ w_1\\ u_1}
   - \GothA^\Nscr \sbm{y_2\\ z_2\\ w_2\\ u_2},\sbm{y_1\\ z_1\\ w_1\\ 
   u_1}-\sbm{y_2\\ z_2\\ w_2\\ u_2}\right\rangle \m=\m \half \|y_1(0)-y_2(0)\|^2-\half \|u_1(0)-u_2(0)\|^2 \]
\[ -\langle \eta_1-\eta_2, w_1-w_2\rangle_{-\half,\half}+\left\langle B_0\left[{\scriptstyle-\half} C_0(w_1-w_2)+u_1(0)-u_2(0)\right],w_1-w_2\right\rangle_{-\half,\half}\m.\]
Using the fact that $B_0^*=C_0$, we move $B_0$ to the right side of the last inner product as $C_0$. We know from \rfb{dom_nonli_A} that $C_0w_1=u_1(0)-y_1(0)$ and $C_0w_2=u_2(0)-y_2(0)$, thus, the above equation becomes 
\[ \hspace{-8mm}\left\langle\GothA^\Nscr\sbm{y_1\\ z_1\\ w_1\\ u_1}
   - \GothA^\Nscr \sbm{y_2\\ z_2\\ w_2\\ u_2},\sbm{y_1\\ z_1\\ w_1\\ 
   u_1}-\sbm{y_2\\ z_2\\ w_2\\ u_2}\right\rangle \m=\m \half \|y_1(0)-y_2(0)\|^2-\half \|u_1(0)-u_2(0)\|^2 \]
\[\hspace{-2mm}-\langle \eta_1-\eta_2, w_1-w_2\rangle_{-\half,\half}+\half\langle u_1(0)-u_2(0)+y_1(0)-y_2(0),u_1(0)-u_2(0)-y_1(0)+y_2(0)\rangle_U\]
\[\hspace{38mm}=-\langle \eta_1-\eta_2, w_1-w_2\rangle_{-\half,\half}\leq 0\m.\]
Using the monotonicity of $\Nscr$, we have proved that the nonlinear (possibly multi-valued)
operator $\GothA^\Nscr$ is dissipative.

Now we show that ${\rm Ran}(I-\GothA^\Nscr)=\Yscr\times H_\half\times
H\times\Uscr$. Considering an arbitrary $\sbm{p & f & g & v}^\top$ in
$\Yscr\times H_{\half}\times H\times \Uscr$, the aim is to show that
the inclusion \vspace{-3mm}
\begin{equation} \label{Onto}
  \left(I-\GothA^\Nscr\right)\bbm{y\\ z\\ w\\ u} \m\ni\m 
  \bbm{p \\ f \\ g \\ v} \vspace{-1mm}
\end{equation}
has solutions $\sbm{y & z & w & u}^\top\in\Dscr(\GothA^\Nscr)\subset
\Hscr^1((-\infty,0);U)\times H_{\half}\times\Dscr(\Nscr)\times
\Hscr^1((0,\infty);U)$. The inclusion \rfb{Onto} can be decomposed
into the following set of equations and an inclusion: \vspace{-2mm}
\begin{equation} \label{Onto_de}
  \left\{\begin{array}{ll} y-\left[\frac{\dd y}{\dd \xi}
  \right]_\Yscr+\delta_0C_0w-\delta_0u(0) \m=\m p, \\[2mm] z-w \m=\m
  f,\\[1mm] A_0 z+w + \half B_0 C_0 w+\Nscr(w)-B_0 u(0)\ni g,\\[1mm] 
  u - \left[ \frac{\dd u}{\dd\xi} \right]_\Uscr \m=\m v.
  \end{array}\right.
\end{equation}
The last part of \rfb{Onto_de} is $\left(I-\left[\frac{\dd}{\dd \xi}
\right]_\Uscr\right) u=v$. Recall that $\left[\frac{\dd}{\dd\xi}
\right]_\Uscr$ is the generator of the left-shift semigroup on
$\Uscr$, with the domain $\Hscr^1((0,\infty);U)$. It is well known
that this semigroup is contractive, hence according to the
Lumer-Phillips theorem (see \cite[Sect.~3.8]{obs_book}) its generator
is maximally dissipative. Consequently, ${\rm Ran}\left(I-\left[
\frac{\dd}{\dd\xi}\right]_\Uscr\right)=\Uscr$ and for any $v\in\Uscr$
there exists a unique $u\in \Hscr^1((0;\infty);U)$ such that the last
line of \rfb{Onto_de} holds.

Substituting $z=f+w$ into the third part of \rfb{Onto_de} and
rearranging, we obtain \vspace{-1.5mm}
\[ \left(A_0+\half B_0 C_0\right)w + (I+\Nscr)(w) \m\ni\m g-A_0
   f+B_0 u(0), \vspace{-0.5mm} \]
where the right-hand side is a given vector in $H_{-\half}$, and we
are searching for $w\in \Dscr(\Nscr)$. Introduce $\tilde{w}=
A_0^{\half}w\in H$ and $\tilde{g}=A_0^{-\half}(g-A_0f+B_0u(0))\in H$.
On applying $A_0^{-\half}$ on both sides of the above inclusion, we
obtain \vspace{-1mm}
\begin{equation} \label{Inclusion3}
  A_0^{-\half}\left(A_0+\half B_0 C_0\right) A_0^{-\half}\tilde{w}
  + A_0^{-1}\tilde{w}+A_0^{-\half}\Nscr \left(A_0^{-\half}\tilde{w}
  \right) \m\ni\m \tilde{g}.\vspace{-1.5mm}
\end{equation}
Note that $A_0^{-\half}$ is an isomorphism from $H$ to $H_\half$ and
also from $H_{-\half}$ to $H$. Hence the operator
$\tilde{\Nscr}=A_0^{-\half}\Nscr A_0^{-\half}$ is maximal monotone on
$H$, with the domain \m $\Dscr(\tilde{\Nscr})=A_0^{\half}\Dscr(\Nscr)$.
Using Rockafellar's theorem (Sect.~\ref{sec2}) we obtain that
$A_0^{-1}+\tilde{\Nscr}$ is a maximal monotone mapping from $H$ to
$H$. The inclusion \rfb{Inclusion3} can be rewritten as \vspace{-1mm}
\begin{equation} \label{Wolves}
   \left(I+\half A_0^{-\half} B_0 C_0 A_0^{-\half}+A_0^{-1} +
   \tilde{\Nscr}\right)\tilde{w} \m\ni\m \tilde{g}.\vspace{-2mm}
\end{equation} 
Since $A_0^{-\half}B_0$ is a bounded mapping from $U$ to $H$, its
adjoint $C_0A_0^{-\half}$ is a bounded mapping from $H$ to $U$.
Therefore, $A_0^{-\half}B_0C_0A_0^{-\half}$ is a bounded and positive
(hence maximal monotone) mapping from $H$ to $H$. Invoking again
Rockafellar's theorem, we obtain that the combined operator $\half
A_0^{-\half}B_0 C_0 A_0^{-\half} + A_0^{-1} + \tilde{\Nscr}$ is 
maximal monotone on $H$. Then, following Minty's theorem (see
Sect.~\ref{sec2}) we conclude that \vspace{-2mm}
\[ {\rm Ran} \left(I+\half A_0^{-\half} B_0 C_0 A_0^{-\half} + 
   A_0^{-1}+\tilde{\Nscr}\right) \m=\m H , \vspace{-0.5mm} \]
and the inclusion \rfb{Wolves} (and hence also \rfb{Inclusion3}) has a
unique solution $\tilde{w}\in\Dscr(\tilde{\Nscr})$ for each $\tilde{g}
\in H$. Consequently, $w=A_0^{-\half}\tilde{w}$ and $z=f+w$ are unique
solutions of the second and third parts of \rfb{Onto_de}. Now the
existence and uniqueness of a solution $y$ of the first part of
\rfb{Onto_de} remains to be verified. For this, it will be helpful to
recall from \rfb{Kirsten_Morris} that for any 
$y\in\Hscr^1((-\infty,0);U)$, \vspace{-1mm}
\begin{equation} \label{Ragnick}
   \left[ \frac{\dd y}{\dd \xi}\right]_\Yscr \m=\m \frac{\dd y}
   {\dd \xi}-\delta_0 y(0).\vspace{-1mm}
\end{equation}

From the definition of the domain of $\Dscr(\GothA^\Nscr)$ as in 
\rfb{dom_nonli_A}, we know that $y(0)=-C_0 w+u(0)$. Hence on 
substituting this equation into the first part of \rfb{Onto_de}, we
obtain $y-\frac{\dd y}{\dd \xi}=p$. Then $y\in\Hscr^1((-\infty,0);U)$
can be obtained by solving this differential equation for the given
$p\in L^2((-\infty,0];U)$ and for a known $y(0)=-C_0w+u(0)$. The 
solution is \vspace{-2mm}
\[ y(\xi) \m=\m e^\xi y(0)-\int^\xi_{0} e^{\xi-\sigma}p(\sigma)\dd
   \sigma \quad\quad \FORALL \xi\leq 0.\vspace{-2mm} \]
Thus, for any $\sbm{p & f & g & v}^\top\in\Yscr\times H_\half\times H
\times\Uscr$, the inclusion \rfb{Onto} has a unique solution \\
$\sbm{y & z & w & u}^\top\in \Dscr(\GothA^\Nscr)$, i.e., the nonlinear
operator $I-\GothA^\Nscr$ is onto. Since $\GothA^\Nscr$ is
dissipative, by Minty's theorem we conclude that $\GothA^\Nscr$ is
maximally dissipative.

We claim that $\Dscr(\GothA^\Nscr)$ is dense in $\Yscr\times H_\half
\times H\times \Uscr$. Consider $\sbm{y & z & w & u}^\top\in\Yscr
\times H_\half\times H\times \Uscr$ and an arbitrary $\e>0$. It is
clear that there exists $u_\e\in\Hscr^1((0,\infty);U)$ such that
$\|u_\e-u\|_\Uscr\leq\frac{\epsilon}{2}$. Since $\Dscr(\Nscr)$ is
dense in $H_\half$ (by assumption), it is also dense in $H$, hence we
can find $w_\e\in\Dscr(\Nscr)$ such that $\|w_\e-w\|_H\leq\frac{\e}
{2}$. Take $\eta_\epsilon\in \Nscr(w_\epsilon)$ and denote 
\vspace{-2mm}
\[ z_0 \m=\m -A_0^{-1}\left[ \half B_0 C_0 w_\e + \eta_\e - B_0
   u_\e(0) \right], \]
so that $z_0\in H_\half$. Then for any $z_1\in H_1$, we have
\vspace{-3mm}
\begin{equation} \label{Stella}
   A_0(z_0+z_1) + \half B_0 C_0 w_\e + \eta_\epsilon - B_0 
   u_\e(0) \m=\m A_0 z_1\in H.\vspace{-1mm}
\end{equation}
Since $H_1$ is dense in $H_\half$, we can choose $z_1$ such that, by
denoting $z_\e=z_0+z_1$, we have that \m $\|z_\e-z\|_{\half}\leq
\frac{\e}{2}$. We see from \rfb{Stella} that $(z_\e, w_\e,u_\e)$
satisfy the nonempty intersection condition in \rfb{dom_nonli_A}. It
remains to find $y_\e$ close to $y$. It is not difficult to verify
that the space $\Hscr^1((-\infty,0);U)$ is dense in $\Yscr$, even if we
impose a fixed boundary value at $\xi=0$. Thus, we can find some $y_\e
\in\Hscr^1((-\infty,0);U)$ satisfying $y_\e(0)=-C_0w_\e+u_\e(0)$ such
that $\|y_\e-y\|_\Yscr\leq\frac{\e}{2}$. Thus, $\sbm{y_\e & z_\e & 
w_\e & u_\e}^\top\in\Dscr(\GothA^\Nscr)$ and \vspace{-2mm}
\[ \left\|\bbm{y_\e\\ z_\e\\ w_\e\\ u_\e}-\bbm{y\\ z\\ w\\ u}\right\|
   ^2= \|y_\e-y\|^2_\Yscr + \|z_\e-z\|^2_{\half} + \|w_\e-w\|^2_H +
   \|u_\e-u\|^2_\Uscr \leq \e^2.\vspace{-2mm} \]
This shows that $\Dscr(\GothA^\Nscr)$ is dense in 
$\Yscr\times H_\half\times H\times\Uscr$.

According to the Crandall-Pazy theorem (Sect.~\ref{sec2}) there is a
unique strongly continuous semigroup of nonlinear contraction
operators $\ULax^\Nscr$ such that $\GothA^\Nscr_0$ is the generator of
$\ULax^\Nscr$, where $\GothA^\Nscr_0\sbm{y & z & w & u}^\top$ is the
element of smallest norm in $\GothA^\Nscr\sbm{y & z & w & u}^\top$ for
every $\sbm{y & z & w & u}^\top\in\Dscr(\GothA^\Nscr)$. Let $\sbm{y_t
& z_t & w_t & u_t}^\top=\ULax^\Nscr_t\sbm{y_0 & z_0 & w_0 & u}^\top$
(for all $t\geq 0$), where $\sbm{y_0 & z_0 & w_0 & u}^\top\in\Dscr
(\GothA^\Nscr)$. It follows from Theorem \ref{Putin} that $\sbm{y_t &
z_t & w_t & u_t}^\top\in\Dscr(\GothA^\Nscr)$ for all $t\geq 0$, and
the function $t\rarrow\sbm{y_t & z_t & w_t & u_t}^\top$ is Lipschitz
continuous and right differentiable in $\Yscr\times H_\half\times H
\times\Uscr$, for $t\in[0,\infty)$. Moreover, the right derivative of
this function is right continuous and its norm (of the derivative) is
nonincreasing. From the second row of the formula of $\GothA^\Nscr$ we
see that $w_t=\dot z_t$. Thus we obtain that $z\in C^1_r([0,\infty);
H_\half)$ and $\dot{z}\in C^1_r([0,\infty);H)$. The statements about
differentiability almost everywhere follow from Proposition 
\ref{Lebesgue}. Now from the first and third rows of the formula of 
$\GothA^\Nscr$ we see that, if we denote $x(t)=\sbm{z_t\\ \dot z_t}$ 
and $y(t)=-C_0\dot z_t+u(t)$, then $(x,u,y)$ is a classical solution
of \rfb{Inclusion1} and \rfb{Output1} and 
$y_t=\Sscr_{-t}(y_0+\PPP_t y)$. 

We can now define the operators $\Sigma^{\rm st}_t$ and $\Sigma^{\rm 
out}_t$ (for $t\geq 0$) such that $x(t)=\Sigma^{\rm st}_t\sbm{x(0)\\
u}$ and $\PPP_t y=\Sigma^{\rm out}_t\sbm{x(0)\\ u}$. These operators
can be defined on all of $X\times\Uscr$ by density and continuous
extension. Now we see that $\ULax^\Nscr_t$ has the structure
\rfb{J_Pasha}, so that it determines a well-posed system
$\Sigma^\Nscr$ (that consists of the operator families $\Sigma^{\rm
st}$ and $\Sigma^{\rm out}$).

We have proved above that $\ULax^\Nscr$ is a contraction semigroup on
$\Yscr\times H_\half\times H\times\Uscr$, therefore for the initial
conditions $\sbm{0 & x_{01} & u_1}^\top,\ \sbm{0 & x_{02} & u_2}^\top
\in\Yscr\times X\times\Uscr$, we have that
\[ \left\|\ULax^\Nscr_\tau \bbm{0\\ x_{01}\\ u_1} - \ULax^\Nscr_\tau
   \bbm{0\\ x_{02}\\ u_2}\right\| \m\leq\m \left\| \bbm{0\\ x_{01}\\ 
   u_1} - \bbm{0\\ x_{02}\\ u_2}\right\| \FORALL \tau\geq 0.\] 
Remember the concepts of state trajectories and output functions of a
well-posed system from Sect.~\ref{sec4}. Let $x_1,x_2$ be state 
trajectories of \m $\Sigma^\Nscr$, as follows: \m $x_1(\tau)=\Sigma
^{\rm st}_\tau\sbm{x_{01}\\ u_1}$, \m $x_2(\tau)=\Sigma^{\rm st}_\tau
\sbm{x_{02}\\ u_2}$. Let $y_{\tau 1}$ and $y_{\tau 2}$ be output
functions of \m $\Sigma^\Nscr$ on the interval $[0,\tau]$, as 
follows: $y_{\tau 1}=\Sigma^{\rm out}_\tau \sbm{x_{01}\\ u_1}$, 
$y_{\tau 2}=\Sigma^{\rm out}_\tau\sbm{x_{02}\\ u_2}$. Then the last
estimate can be rewritten as \vspace{-1mm}
\[ \int_0^\tau \|y_{\tau 1}(t)-y_{\tau 2}(t)\|_U^2\dd t + \|x_1(\tau)
   - x_2(\tau)\|_X^2 + \int_\tau^\infty \|u_1(t)-u_2(t)\|_U^2\dd t
   \vspace{-2mm}\]
\[ \m\hspace{2cm}\leq\m \|x_{01}-x_{02}\|_X^2 + \int_0^\infty
   \|u_1(t)-u_2(t)\|_U^2\dd t \m.\]
Here we have used that the first component of $\ULax^\Nscr_\tau[0\
x_{01}\ u_1]^\top$ is $\Sscr_{-\tau}y_{\tau 1}$, similarly for
$x_{02}$ and $u_2$, and that $\Sscr_{-\tau}$ is unitary. Equivalently,
our estimate becomes \rfb{energy_bal_M}. Therefore, $\Sigma^\Nscr$ is 
an incrementally scattering passive well-posed system. $\quad\square$

\smallskip

\begin{remark} \label{Messi}
Note that if $(x,u,y)$ is a classical solution of \rfb{Inclusion1}
and \rfb{Output1}, as described in Theorem \ref{Main}, then the fact
that $\ddot{z}(t)\in H$ for all $t\geq 0$ implies that the
compatibiltiy condition \rfb{Haaland} remains valid along the entire
solution: \vspace{-1mm}
\[ \left(A_0 z(t) + \half B_0 C_0\dot{z}(t) + \Nscr(\dot{z}(t)) - 
   Bu(t)\right)\cap H \m\neq\m \emptyset \quad \FORALL t\geq 0.
   \vspace{-1mm} \] 
\end{remark}

\begin{remark} \label{Barbu}
In \cite{Barbu}, V.~Barbu has investigated evolution equations that
are closely related to \rfb{Inclusion1}. Consider the real Hilbert
spaces $V$, $H$, where $V$ is dense and continuously embedded in $H$,
such that together with the dual space $V'$, we have the Gelfand
triple $V\subset H\subset V'$. The Cauchy problem examined in
\cite[Sect.~5.6]{Barbu} is \vspace{-1mm}
\begin{equation} \label{Bar}
   \left\{\begin{array}{ll} \ddot{z}(t) + A z(t) + B(\dot{z}(t)) 
   \m\ni\m f(t)\quad \FORALL t\in [0,T],\\ [2mm] z(0)\m =\m z_0, 
   \quad \dot{z}(0) \m=\m w_0,\end{array}\right. \vspace{-1mm}
\end{equation}
where $A$ is a positive operator from $V$ to $V'$, and $B$ is a
nonlinear (possibly set-valued) maximal monotone operator from $V$ to
$V'$ (its domain is contained in $V$). In \cite[Sect.~5.6]{Barbu} it
is shown that if the input function $f\in W^{1,1}([0,T];H)$ and the
initial conditions $z(0)=z_0\in V$ and $\dot{z}(0)=w_0\in \Dscr(B)$
are such that $(Az_0+Bw_0)\cap H\neq\emptyset$, then \rfb{Bar} has a
solution $z$ with the smoothness \vspace{-2mm}
\[ z\in W^{1,\infty}([0,T];V)\cap W^{2,\infty}([0,T];H)
   \vspace{-0.3mm}.\]
To facilitate the comparison between our results and those in
\cite{Barbu}, we give a table showing how the notation here is related
to the notation in \cite{Barbu}. \vspace{1mm}
\begin{center}\begin{tabular}{|c|c|c|c|c|c|} 
 \hline
 Barbu's notation & $V$ & $V'$ & $A$ & $B$ & $f$ \\  
 \hline 
 our notation & $H_\half$ & $H_{-\half}$ & $A_0$ & $\half B_0 C_0 +
 \Nscr $ & $B_0u$ \\
 \hline
\end{tabular}
\vspace{1mm}
\end{center}
In \cite{Barbu}, the input function $f$ is assumed to be in $W^{1,1}
([0,T];H)$ (i.e., absolutely continuous), while we consider $B_0 u$ to
be in the larger space $W^{1,2}((0,\infty);V')$, i.e., $B_0 u\in
\Hscr^1((0,\infty);H_{-\half})$. Thus, our control operator $\sbm{0\\
B_0}$ may be unbounded, while the hypotheses in \cite{Barbu} imply
that the control operator is $\sbm{0 \\ I}$, which maps boundedly into
the state space $V\times H$. In \cite{Barbu}, the solution $z$ of
\rfb{Bar} is Lipschitz continuous with values in $V$ (or $H_\half$ in
our notation) and $\dot{z}$ is Lipschitz continuous with values in
$H$. This is consistent with our definition of the classical solutions
of \rfb{Inclusion1}. Our formulation of the nonlinear system
(described by \rfb{Inclusion1} and \rfb{Output1}) allows unbounded
(for instance, boundary) observation, while in \cite{Barbu}, outputs
are not considered and only the state trajectories of the evolution
equation \rfb{Bar} are investigated.
\end{remark} 

%%%%%%%%%%**********%%%%%%%%%%**********%%%%%%%%%%**********%%%%%%%%%%
\section{Systems with nonlinear boundary damping} \label{sec6}
% Section 6

In this section we investigate the well-posedness of a class of
nonlinear systems with boundary control and boundary observation,
perturbed by a nonlinear damping term acting on the boundary. Recall
the concept of hemi-continuous operator, defined in Sect.~\ref{sec2}.
Recall also that in this article, by $\ddot{z}$ and $\dot{z}$ we mean
the right second derivative and the right derivative of $z$,
respectively. Using the assumptions and notation of Sect.~\ref{sec1}
and Sect.~\ref{sec3}, in particular the space $Z_0$ from \rfb{Matty}, 
we now consider nonlinear systems described by the second order
differential equation \vspace{-2mm}
\begin{equation} \label{BDamp}
  \ddot{z}(t) + \half B_0 C_0 \dot{z}(t) + B_0\Nscr_0 \left( C_0
  \dot{z}(t)\right)+A_0z(t) \m=\m B_0 u(t), \vspace{-2mm}
\end{equation}
and the output equation \vspace{-3mm}
\begin{equation} \label{OutputBD}
  y(t) \m=\m -C_0\dot{z}(t) + u(t) \vspace{-1.2mm},
\end{equation}
where $\Nscr_0:U\rarrow U$ is a single-valued monotone hemi-continuous
nonlinear operator with $\Dscr(\Nscr_0)=U$. Notice that \rfb{BDamp} is
similar to \rfb{Inclusion1} with $B_0\Nscr_0\left(C_0\dot{z}(t)
\right)$ replacing $\Nscr(\dot{z}(t))$ and with an equality instead of
the inclusion (as $\Nscr_0$ is a single-valued operator).

The proposition below shows how systems described by \rfb{BDamp} and
\rfb{OutputBD} may arise from systems with boundary control and
boundary observation. The operators $G_0$ and $C_0$ that will appear
here are typically boundary trace operators, so that the damping
operator $\Nscr_0$ appears in the equations describing what happens on
the boundary.

This proposition is the nonlinear extension of Theorem 1.4 in 
\cite{WeTu03}.

{\color{blue}
\begin{proposition} \label{CR7}
On the Hilbert space $Z_0$ (as defined in Sect.{\rm\ref{sec3}}),
assume that there exists an operator $G_0\in\Lscr(Z_0,U)$ such that
$G_0 H_1=\{0\}$ and $G_0\m A_0^{-1}B_0=I$. For all $z\in Z_0$, define
$L_0z=A_0 z-B_0 G_0 z$ on $Z_0$ (here we have used the extension of 
$A_0$ to $H_\half$). Then ${\rm Ker}(G_0)=H_1$, $L_0\in\Lscr(Z_0,H)$
and the system $\Sigma^\Nscr$ defined by the differential equation
\rfb{BDamp} and the output \rfb{OutputBD} can also be described by the
equations \vspace{-1mm}
\begin{equation} \label{BSC}
   \left\{\begin{array}{ll} \ddot{z}(t)+L_0z(t)=0,\vspace{1mm}\\ 
   G_0 z(t) + \half C_0 \dot{z}(t) + \Nscr_0 \left( C_0\dot{z}(t)
   \right) \m=\m u(t),\vspace{1mm}\\ G_0 z(t) - \half C_0 \dot{z}(t)
   + \Nscr_0 \left( C_0\dot{z}(t) \right) \m=\m y(t),
   \end{array} \right.\vspace{-2mm}
\end{equation} 
in the following sense:
\begin{itemize}
\item[{\rm (1)}] If $u\in\Hscr^1((0,\infty);U)$, $z_0\in Z_0$ and $w_0
\in H_\half$, then the compatibility condition \vspace{-4mm}
\[ G_0z_0+\half C_0w_0+\Nscr_0\left(C_0w_0\right)=u(0)\vspace{-3mm}\]
is equivalent to \rfb{Haaland} with $\Nscr=B_0\Nscr_0 C_0$, or in 
other words\vspace{-2mm}
\[ A_0 z_0+\half B_0 C_0 w_0 + B_0\Nscr_0\left( C_0 w_0\right) -B_0
   u(0)\in H.\vspace{-2mm}\] 
\item[{\rm (2)}] If $z(t)\in Z_0$, $\dot z(t)\in H_\half$ and $\ddot
z(t)\in H$, then \rfb{BSC} is equivalent to \rfb{BDamp} and
\rfb{OutputBD}. \end{itemize}
\end{proposition}}

{\em Proof.} \m If $z\in Z_0$ then it follows from the definition of
$Z_0$ that $z=z_1+A_0^{-1}B_0v$, where $z_1\in H_1$ and $v\in U$.
Applying $G_0$ on both sides we obtain\vspace{-2mm}
\[ G_0 z \m=\m G_0 z_1+G_0 A_0^{-1} B_0 v \m=\m v \m.\vspace{-2mm}\]
Therefore, ${\rm Ker}(G_0)=H_1$. On applying $G_0A_0^{-1}$ on both 
sides of the definition of $L_0z$, we have \vspace{-2mm}
\[ G_0 A_0^{-1} L_0z \m=\m G_0 A_0^{-1} A_0 z - G_0 A_0^{-1} B_0 
   G_0 z \m=\m G_0 z - G_0 z \m=\m 0.\vspace{-1mm}\]
Hence, $A_0^{-1}L_0z\in {\rm Ker}(G_0)=H_1$, and it follows that
$L_0z\in H$. Since $L_0\in \Lscr(Z_0, H_{-\half})$, by the closed
graph theorem we obtain that $L_0\in\Lscr(Z_0,H)$. 

Now we prove part $(1)$ of the proposition. If the condition
\vspace{-2mm}
\[ A_0 z_0 + \half B_0 C_0 w_0 + B_0\Nscr_0 \left( C_0 w_0\right)
   - B_0 u(0) \in H \vspace{-2mm}\] 
holds, then clearly $z_0\in Z_0$. Applying $G_0A_0^{-1}$ on both 
sides, we obtain \vspace{-2mm}
\[ \hspace{-0.5cm}G_0 z_0 + \half C_0 w_0 + \Nscr_0 \left( C_0 w_0
   \right) - u(0) \m=\m 0,\vspace{-1mm}\]
since $G_0H_1=\{0\}$. Conversely, if \m $G_0z_0+\half C_0 w_0 + 
\Nscr_0\left(C_0w_0\right)=u(0)$ holds, then on applying $B_0$ on 
both sides we obtain \vspace{-2mm}
\[ B_0 G_0 z_0 + \half B_0 C_0 w_0 + B_0\Nscr_0 \left( C_0 w_0 
   \right) \m=\m B_0 u(0).\vspace{-1mm}\] 
Using the definition of $L_0$, we have that $B_0 G_0 z_0=A_0 z-
L_0 z_0$, thus \vspace{-1mm}
\begin{equation} \label{Schengen}
   A_0 z_0 + \half B_0 C_0 w_0 + B_0 \Nscr_0 \left( C_0 w_0\right)
   - B_0 u(0) \m=\m L_0 z_0 \m\in\m H \m.\vspace{-1mm}
\end{equation}

We move on to part (2). To show that \rfb{BDamp} and \rfb{OutputBD}
imply \rfb{BSC}, first notice that, due to $\ddot z(t)\in H$,
\rfb{BDamp} implies $A_0z(t)+\half B_0 C_0\dot z(t)+B_0\Nscr_0(C_0\dot
z(t))- B_0 u(t)\in H$. By a reasoning similar to the one used to
derive \rfb{Schengen}, we can show that actually, \m $A_0z(t)+\half 
B_0 C_0\dot z(t)+B_0\Nscr_0(C_0\dot z(t))-B_0 u(t)=L_0 z(t)$. This
with \rfb{BDamp} implies the first equation in \rfb{BSC}.

Now apply $G_0 A_0^{-1}$ on both sides of \rfb{BDamp}, use that
$\ddot z(t)\in H$, so that $G_0 A_0^{-1}\ddot z(t)=0$, and also use 
that $G_0 A_0^{-1}B_0=I$. This shows that the second equation in 
\rfb{BSC} holds. Combining this with \rfb{OutputBD}, we obtain the 
third equation in \rfb{BSC}.

To show that \rfb{BSC} implies \rfb{BDamp}, use the first equation in
\rfb{BSC} combined with \rfb{Schengen}. Finally, subtracting the
second equation in \rfb{BSC} from the third, we get \rfb{OutputBD}. 
$\quad\square$

\smallskip

\begin{remark}\label{z_norm}
Under the assumptions of Proposition \ref{CR7}, the norm on $Z_0$
(defined in \rfb{Matty}) can be expressed also as
follows: \vspace{-2mm}
\begin{equation} \label{fusion}
   \|z\|^2_{Z_0} \m=\m \|L_0z\|^2_H+\|G_0z\|^2_U.
\end{equation} 
Indeed for a given $z\in Z_0$, a unique $v\in U$ can be determined
using the formula $G_0 z=v$ (as shown in the beginning of the above
proof). For $z=z_1+A^{-1}_0B_0v$, the norm on $Z_0$ is
\[ \|z\|^2_{Z_0} \m=\m \|z_1\|^2_{H_1}+\|v\|^2_U \m=\m \|z-A^{-1}_0
   B_0v\|^2_{H_1}+\|G_0z\|^2_U \m=\m \|A_0z-B_0G_0z\|^2_{H} +
   \|G_0z\|^2_U. \]
Since $A_0-B_0 G_0=L_0$, we obtain \rfb{fusion}. 
\end{remark}

{\color{blue}
\begin{corollary} \label{WellposedBSC}
Assume that $u\in\Hscr^1((0,\infty);U)$ and the initial state
$\sbm{z_0\\ w_0}\in H_\half\times H_\half$ satisfy \vspace{-5mm}
\[ A_0z_0+\half B_0 C_0 w_0 + B_0\Nscr_0 \left( C_0 w_0\right) - 
   B_0u(0) \in H, \quad y(0) = -C_0w_0+u(0), \]
where $\Nscr_0:U\rarrow U$ is a nonlinear monotone hemi-continuous
operator with $\Dscr(\Nscr_0)=U$. Then \rfb{BDamp} and \rfb{OutputBD}
have a unique classical solution $(x,u,y)$ such that $x(0)=\sbm{z(0)\\
\dot z(0)}=\sbm{z_0\vspace{1.5mm}\\ w_0}$. The classical solutions of \rfb{BDamp}
and \rfb{OutputBD} satisfy that $z(t)\in Z_0$ and \rfb{BSC} holds for
all $t\geq 0$.

Moreover, these classical solutions have all the properties stated in
Theorem {\rm\ref{Main}}, in particular the continuity properties
\rfb{Well-posed} and the energy balance inequality \rfb{energy_bal_M}.
Hence, \rfb{BDamp} and \rfb{OutputBD} determine an incrementally
scattering passive well-posed system $\Sigma^\Nscr$ with the state
space $X=H_\half\times H$, and with the input and output space $U$.
\end{corollary}}

{\em Proof.} \m If we denote $\Nscr = B_0\Nscr_0 C_0$, then
\rfb{BDamp} becomes \rfb{Inclusion1}, with equality instead of the
inclusion (because $\Nscr$, like $\Nscr_0$, is single-valued). The
operator $\Nscr$ is maximal monotone. Indeed, it is easy to see that
$\Nscr$ is monotone and hemi-continuous from $H_\half$ to
$H_{-\half}$, so that it is maximal monotone according to Proposition
\ref{Hemicont}. Therefore, the system $\Sigma^\Nscr$ described by
\rfb{BDamp} and \rfb{OutputBD} satisfies the conditions in Theorem
\ref{Main}, hence it is well-posed and it has all the properties
stated in Theorem \ref{Main}. The fact that $z(t)\in Z_0$ for all
$t\geq 0$ is easy to see from \rfb{BDamp} and from the fact that
$\ddot z(t)\in H$ (as stated in Theorem \ref{Main}). The fact that
classical solutions satisy \rfb{BSC} for all $t\geq 0$ follows from
Proposition \ref{CR7}. $\quad \square$

%%%%%%%%%%**********%%%%%%%%%%**********%%%%%%%%%%**********%%%%%%%%%%
\section{Nonlinear damping for the wave equation on a bounded domain}
\label{sec7} % Section 7

In this section we apply our earlier results (from Sections 5 and 6)
to the wave equation with nonlinear distributed or boundary damping.
We formulate such systems that fit into our framework and then show
that (for suitable initial and boundary conditions) they have unique
classical and generalized solutions, and these satisfy the energy
balance inequality \rfb{energy_bal_M}. For the spaces and the operators,
we follow Sect.~7 of \cite{WeTu03}, which considers the linear case.
We study the wave equation on an $n$-dimensional domain $\Om$, with a
control input $u$ applied on a part $\Gamma_1$ of the boundary (called
the active boundary) and with the output $y$ observed (measured) on
the same $\Gamma_1$, while the other part of the boundary just
reflects the waves. One may regard $u$ as the incoming wave and $y$ as
the outgoing wave. The one-dimensional version of the example with
distributed cubic damping is in our conference paper
\cite{SW:CDC2022}.

Consider $H=L^2(\Om)$, where $\Om\subset\rline^n$ is open and bounded,
with Lipschitz boundary $\Gamma$. Let $\Gamma_0$ and $\Gamma_1$ be
nonempty open and disjoint subsets of $\Gamma$, such that $\Gamma_0
\cup\Gamma_1$ is dense in $\Gamma$. Assume that the boundaries of
$\Gamma_0$ and $\Gamma_1$ have surface measure zero in $\Gamma$. The
input and output space is $U=L^2(\Gamma_1)$. Assume that $b\in
L^\infty(\Gamma_1)$ (real valued) and $b(x)\neq 0$ for almost every
$x\in\Gamma_1$. The linearly damped wave equation on $\Om$, with
boundary input and boundary observation, is described by the following
set of equations: \vspace{-2mm}
\begin{equation} \label{wave}
   \left\{ \begin{array}{ll} \ddot z(x,t) \m=\m \Delta z(x,t) &
   \mbox{on }\Om\times[0,\infty),\\ z(x,t) \m=\m 0 \bluff & \mbox
   {on }\Gamma_0\times [0,\infty), \\ \frac{\partial }{\partial \nu}
   z(x,t) +b(x)^2\dot{z}(x,t) \m=\m \sqrt{2} b(x) u(x,t) \bluff & 
   \mbox{on }\Gamma_1 \times [0,\infty),\\ \frac{\partial }{\partial
   \nu} z(x,t) - b(x)^2\dot{z}(x,t) \m=\m\sqrt{2}b(x)y(x,t)\bluff
   & \mbox{on }\Gamma_1 \times [0,\infty). \end{array} \right.
\end{equation}
We normally impose initial conditions as follows: \m $z(x,0)=z_0(x),
\ \dot z(x,0)=w_0(x)$ (for $x\in\Om$) where the functions $z_0$ and 
$w_0$ are the initial states (in spaces to be specified). The term 
$b(\cdot)^2\dot{z}(\cdot,t)$ is the viscous damping acting on the 
boundary $\Gamma_1$. 

We denote by $\gamma$ the {\it Dirichlet trace operator}, so that
\vspace{-2mm}
\[ \gamma g \m=\m g|_{\Gamma} \FORALL g\in C(\overline\Om) \m. \]
It is well known that $\gamma$ can be extended to a bounded operator 
from $\Hscr^1(\Om)$ to $L^2(\Gamma)$. Denoting by $\Rscr$ the 
restriction operator mapping $L^2(\Gamma)$ onto $L^2(\Gamma_1)$, we
define \vspace{-2mm}
\[ \gamma_0 g \m=\m \Rscr \gamma g  \FORALL g\in\Hscr^1(\Om) \m. \]
Using this notation, we introduce a Hilbert space $\Hscr^1_{\Gamma_0}
(\Om)$ as follows: \vspace{-2mm}
\[ \Hscr^1_{\Gamma_0}(\Om) \m=\m \{g\in \Hscr^1(\Om)\ \ | \ \ 
   (I-\Rscr)\gamma g=0\},\quad \|g\|_{\Hscr^1} \m=\m \|\nabla 
   g\|_{(L^2(\Om))^n}\m\m. \vspace{-2mm}\]
In other words, $\Hscr^1_{\Gamma_0}(\Om)$ is the space of all those
functions in $\Hscr^1(\Om)$ that vanish on $\Gamma_0$. We also 
introduce the Neumann trace operator on $\Gamma_1$ as follows:
\vspace{-2mm}
\[ \gamma_1 f \m=\m \frac{\partial}{\partial\nu}f\big|_{\Gamma_1}
   \m=\m \langle\nabla f,\nu\rangle \FORALL f\in C^1(\overline\Om)
   \m, \] 
where $\nu$ is the unit vector in the outward normal direction to $\Gamma$. It is known (see,
for instance, \cite[Sect.~7]{WeTu03}) that $\gamma_1$ can be extended
to those $f\in\Hscr^1_{\Gamma_0}(\Om)$ for which $\Delta f\in L^2
(\Om)$, and then $\gamma_1 f$ is in a certain Sobolev space on 
$\Gamma_1$ that includes $L^2(\Gamma_1)$ densely. We introduce the
space \vspace{-1mm}
\begin{equation} \label{Z_0}
  Z_0 \m=\m \left\{f\in \Hscr^1_{\Gamma_0}(\Om) \left|\m\m \Delta 
  f\in L^2(\Om),\m\m\gamma_1 f\in bL^2(\Gamma_1)\right.\right\}, 
  \vspace{-1mm} \end{equation}
We define $A_0=-\Delta$ (in the sense of distributions), with the
domain \vspace{-1.5mm}
\begin{equation} \label{Andropov}
   \Dscr(A_0) \m=\m H_1 \m=\m \left\{ z\in Z_0\m\m \left|\ \ \gamma_1
   z=0\right. \right\} \m.\vspace{-1mm}
\end{equation}
Then $H_\half=\Dscr(A_0^{\half})=\Hscr^1_{\Gamma_0}(\Om)$. The
operator $A_0$ is positive and boundedly invertible, and its extension
mapping $H_\half$ to $H_{-\half}$ is also denoted by $A_0$. For
technical details about these spaces and the operators below we refer
to \cite[Sect.~7]{WeTu03} (see also \cite{Kurula_Zwart,Skrepek},
\cite[Sect.~3.7]{obs_book}).

We define the Neumann map \m $N\in\Lscr(U,H_\half)$ such that: $Nu=g$
if and only if $g\in\Hscr^1_{\Gamma_0}(\Om)$, $\Delta g=0$ and
$\gamma_1 g=u$. It is shown, for instance, in \cite[Sect.~7]{WeTu03}
that indeed such an operator exists and, moreover, $N^*A_0=\gamma_0$.
Clearly $\gamma_1 N=I$. The following proposition is an easy
consequence of the facts stated so far in this section.

{\color{blue} \begin{proposition} \label{BCS_operators}
Define the operators $C_0\in\Lscr(H_\half,U)$,
$B_0\in\Lscr(U,H_{-\half})$ (where $H_{-\half}$ is the dual of 
$H_\half$) and $G_0\in\Lscr(Z_0,U)$ by \vspace{-3mm}
\begin{equation} \label{Gorbachov}
   C_0 \m=\m \sqrt{2}\m b N^* A_0 \m=\m \sqrt{2} \cdot b\gamma_0 \m,
   \qquad B_0 \m=\m C_0^*=\sqrt{2}\m A_0 Nb, \qquad G_0 \m=\m \frac{1}
   {\sqrt{2}}\m b^{-1} \gamma_1 \m.\vspace{-2mm}
\end{equation}
Then $G_0 H_1=\{0\}$ and $G_0\m A_0^{-1}B_0=I$. For these operators
$A_0$ and $B_0$, the space $Z_0$ from \rfb{Matty} coincides with $Z_0$
from \rfb{Z_0}. According to Proposition {\rm\ref{CR7}} we can define
$L_0=A_0-B_0 G_0$ and then $L_0\in\Lscr(Z_0, H)$. We have that $L_0
z=-\Delta\m z$, for every $z\in Z_0$, and the norm on $Z_0$ is given by
\rfb{fusion}, i.e.,\vspace{-2mm}
\[ \|z\|^2_{Z_0} \m=\m \|\Delta\m z\|^2_{L^2(\Om)}+\half\left\| 
   \frac{1}{b}\gamma_1 z\right\|^2_{L^2(\Gamma_1)}\m.\vspace{1mm} \]
\end{proposition}}

According to Proposition \ref{CR7} (the linear version), the linear
wave system \rfb{wave} can be written in the form \rfb{Inclusion1} and
\rfb{Output1}, with $\Nscr=0$ and with the operators $B_0$ and $C_0$
as defined above. We know from \cite{WeTu03} that this wave system is
well-posed and in fact, it is a {\em conservative system}, which means
that the operators $\Sigma_\tau$ from \rfb{Sig4b} are unitary. (In
particular, this implies that \rfb{wellposed_ineq} holds with equality
and with $c_\tau=1$.) We remark that the corresponding undamped wave
equation (i.e., \rfb{wave} but with the third and fourth lines
modified to $\frac{\partial}{\partial\nu}z(x,t)=b(x)u(t)$ and $b(x)
\dot{z}(x,t) |_{\Gamma_1}=y(x,t)$, respectively) is impedance passive
(see \cite{StWe12}) but not well-posed, see \cite[p.~220] {Avdonin}
for details. (The analysis in \cite{Avdonin} shows that if $\Om$ is a
rectangle in $\rline^2$, the control operator corresponding to $B_0$
is not admissible for the undamped wave equation.)

{\bf Wave equation with distributed cubic damping.} \m 
In the following proposition we investigate the effect of a nonlinear
distributed damping in addition to the linear boundary damping already
present in \rfb{wave}. We introduce the nonlinear (``cubic'') operator
$\Nscr_c$ \m defined in $H_\half$, with values in $H_{-\half}$, as
follows: \vspace{-3mm}
\begin{equation} \label{Odessa}
   \Nscr_c(w) \m=\m w^3 \FORALL w\in\Dscr(\Nscr_c) \m=\m 
   \left\{f\in\Hscr^1_{\Gamma_0}(\Om)\m|\ f^3\in \left(
   \Hscr^1_{\Gamma_0}(\Om)\right)'\right\}\m. \vspace{-2mm}
\end{equation}
Thus, if $w\in\Hscr^1_{\Gamma_0}(\Om)$, then $w\in\Dscr(\Nscr_c)$ if
and only if $w^3$ is a bounded linear functional on 
$\Hscr^1_{\Gamma_0}(\Om)=H_\half$, i.e., if there exists a $K>0$ such
that \vspace{-3mm}
\begin{equation} \label{Mexico}
   \left| \int_\Om w^3(\xi) \cdot \phi(\xi)\dd\xi \right| \m\leq\m K 
   \cdot \|\phi\|_{\Hscr^1} \FORALL \phi\in \Hscr^1_{\Gamma_0}(\Om).
\end{equation}
It is easy to see that $\Nscr_c$ is a monotone mapping from $H_\half$
to $H_{-\half}$.

{\color{blue} \begin{lemma} \label{Sobolev_embedding}
If $\Om\subset\rline^n$, where $n\leq 4$, then $\Dscr(\Nscr_c)=
\Hscr^1_{\Gamma_0}(\Om)$.
\end{lemma}}

{\em Proof.} \m For $n=1$ we have that $\Hscr^1(\Om)\subset
C(\bar{\Om})\subset L^p(\Om)$ for any $p\in[1,\infty]$, while for
$n=2$ we have that $\Hscr^1(\Om)\subset L^p(\Om)$ (with continuous
embedding) for any $p\in[1,\infty)$, see \cite[p.~66]{Necas} for the 
last statement. For $n\geq 3$ we can use the Sobolev embedding theorem
(see \cite[p.~63] {Necas}) to see that $\Hscr^1(\Om)\subset L^P(\Om)$
(with continuous embedding), where $\frac{1}{p}=\half-\frac{1}{n}$. 
For $n=3$ we get $p=6$, for $n=4$ we get $p=4$, for $n=5$ we get 
$p=10/3$ while for $n\geq 6$ we get $p\leq 3$. Of course, the 
continuous embedding $\Hscr^1(\Om)\subset L^p(\Om)$ means that there
exists a $K_e>0$ such that \vspace{-3mm}
\begin{equation} \label{Mattia}
   \|\phi\|_p \m\leq\m K_e \|\phi\|_{\Hscr^1} \FORALL \phi\in
   \Hscr^1_{\Gamma_0}(\Om) \m.\vspace{-2mm}
\end{equation}

According to H\"{o}lder's inequality, for any $w\in\Hscr^1(\Om)$,
\vspace{-1.5mm}
\begin{equation} \label{Sebastien}
   |\langle w^3, \phi\rangle| \m\leq\m \|w^3\|_{\bar p}
   \|\phi\|_p \quad \FORALL \phi\in\Hscr^1(\Om),\vspace{-1mm}
\end{equation}
where $\frac{1}{p}+\frac{1}{\bar{p}}=1$, as long as $w^3\in 
L^{\bar{p}}(\Om)$. We see that this is the case for every $w\in
\Hscr^1(\Om)$ if $p/3\geq\bar{p}$, which is the case for $1\leq n\leq
4$, but not for $n\geq 5$. Thus, for $1\leq n\leq 4$, \rfb{Sebastien}
holds for any $w\in\Hscr^1(\Om)$. This, together with \rfb{Mattia}, 
implies \rfb{Mexico}, so that $\Dscr(\Nscr_c)=
\Hscr^1_{\Gamma_0}(\Om)$. $\quad\square$

\smallskip

The following proposition is in the same spirit as Proposition
\ref{CR7}. It shows that the equations \rfb{Inclusion1}, \rfb{Output1}
and also the compatibility condition \rfb{Haaland} can be rewritten,
in the case of the wave equation with cubic damping, in a form that is
closer to a PDE with boundary control and boundary observation. This
system fits into the framework of Theorem \ref{Main} if $n\leq 4$,
because then $\Nscr_c$ is maximal monotone. In Proposition
\ref{Airmax} we shall see that for $n\leq 3$, the equations of the 
system can be written in an even simpler form.  

{\color{blue} \begin{proposition} \label{wave_nl_pro} Recall the
operators $A_0$, $B_0$, $C_0$ and $\Nscr_c$ introduced between
{\rm\rfb{Andropov}} and {\rm\rfb{Odessa}}. If $u\in\Hscr^1
((0,\infty);U)$ and $\sbm{z_0\\ w_0}\in H_\half\times\Dscr(\Nscr_c)$,
then the compatibility condition \rfb{Haaland} (with $\Nscr_c$ in
place of $\Nscr$) takes the form \vspace{-2mm}
\begin{equation} \label{Ben}
   \left(A_0 z_0+\half B_0 C_0 w_0 + \Nscr_c(w_0) - B_0 u(0)\right) 
   \in H . \vspace{-2mm}
\end{equation}
This is equivalent to \vspace{-2mm}
\begin{equation} \label{Biden}
   z_0 + A_0^{-1}\Nscr_c(w_0)\in Z_0\quad \&\quad G_0 \left[ z_0 + 
   A_0^{-1}\Nscr_c(w_0) \right] + \half C_0 w_0 
   \m=\m u(0) \m.\vspace{-0.5mm}
\end{equation}
This condition is further equivalent to \vspace{-1mm}
\begin{equation} \label{Hila}
   z_0 + A_0^{-1}\Nscr_c(w_0) \in Z_0\quad\&\quad \gamma_1
   \left[ z_0 + A_0^{-1}\Nscr_c(w_0)\right] + b^2 \gamma_0 w_0 \m=\m
   \sqrt{2} bu(0) \m.\vspace{-2mm}
\end{equation}

If $\dot z(t)\in H_\half$, $z(t)+A_0^{-1}\Nscr_c(\dot z(t))
\in Z_0$, and $\ddot z(t)\in H$, then the equations\vspace{-2mm}
\begin{equation} \label{Halmos}
   \left\{ \begin{array}{l} \ddot z(t) \m=\m -L_0 \left[ z(t) +
   A_0^{-1} \Nscr_c(\dot{z}(t)) \right], \\ G_0 \left[ z(t) + A_0^{-1}
   \Nscr_c(\dot z(t)) \right] + \half C_0 \dot{z}(t) \m=\m u(t),\bluff
   \\ G_0 \left[ z(t) + A_0^{-1}\Nscr_c (\dot z(t)) \right] - \half
   C_0 \dot{z}(t) \m=\m y(t),\bluff \end{array} \right.\hspace{-2mm}\m 
   \vspace{-2mm}
\end{equation}
are equivalent to \rfb{Inclusion1} and \rfb{Output1}, with $=$ in
place of $\in$ and with $\Nscr_c$ in place of $\Nscr$.

If the dimension $n\leq 4$, then $\Nscr_c$ is maximal monotone. Hence,
if the data $z_0,w_0,u$ satisfy the equivalent compatibility
conditions \rfb{Ben}-\rfb{Hila}, then the set of equations
\rfb{Halmos} has a unique classical solution $(\sbm{z\\ \dot z},u,y)$
such that $z(0)=z_0$, $\dot{z}(0)=w_0$. This classical solution has
all the properties listed in Theorem {\rm\ref{Main}}, in particular,
\rfb{Well-posed}. Moreover, the equations \rfb{Halmos} determine an
incrementally scattering passive well-posed system \m $\Sigma^\Nscr$.
\end{proposition}}

{\em Proof.} \m Recall the space $Z_0$ from \rfb{Matty} and \rfb{Z_0},
the operator $G_0$ from \rfb{Gorbachov}, and the fact that $G_0 
A_0^{-1}B_0=I$. If the condition \rfb{Ben} holds, then it is easy to
see that $z_0+A_0^{-1}\Nscr_c(w_0)\in Z_0$. Applying $G_0 A_0^{-1}$ on
both sides of \rfb{Ben}, and using that $G_0 H_1=\{0\}$, we obtain
\rfb{Biden}. Recalling the definition of $G_0, C_0$ and multiplying \rfb{Biden}
with $\sqrt{2}b$, we obtain \rfb{Hila}. 

Conversely, if \rfb{Hila} holds, then on applying $B_0 b^{-1}$ on 
both sides, we obtain \vspace{-2mm}
\[ B_0 G_0 \left[ z_0 +A_0^{-1}\Nscr_c(w_0) \right] + \half B_0 C_0 
   w_0 \m=\m B_0 u(0) \m.\vspace{-2mm}\] 
From the definition of $L_0$ we have that $B_0 G_0=A_0-L_0$, thus 
\vspace{-2mm}
\[ A_0 z_0 + \Nscr_c(w_0) + \half B_0 C_0 w_0 - B_0 u(0) \m=\m L_0 
   \left[ z_0+A_0^{-1}\Nscr_c(w_0)\right] \m\in\m H \m,\vspace{-2mm}
\]
which clearly implies \rfb{Ben}. 

To show that \rfb{Inclusion1} and
\rfb{Output1} imply \rfb{Halmos}, first we rewrite \rfb{Inclusion1}:
\vspace{-2mm}
\begin{equation} \label{Schengen_bis}
\ddot{z}(t) + A_0 \left[ z(t) + A_0^{-1} \Nscr_c(\dot{z}(t))\right]
   + \half B_0C_0\dot{z}(t) \m=\m B_0 u(t) \FORALL t\geq 0 \m.
   \vspace{-2mm}
\end{equation}
If we apply $G_0 A_0^{-1}$ to both sides and use that $\ddot z(t)\in 
H$, $G_0H_1=\{0\}$, we get : \vspace{-2mm}
\[ G_0 \left[ z(t) + A_0^{-1} \Nscr_c(\dot{z}(t)) \right] + \half G_0
   A_0^{-1}B_0 C_0\dot{z}(t) \m=\m G_0 A_0^{-1}B_0 u(t) \FORALL 
   t\geq 0 \m. \vspace{-2mm}\]
Using that $G_0 A_0^{-1}B_0=I$, we obtain the second equation in 
\rfb{Halmos}.

From \rfb{Schengen_bis}, using that $A_0=L_0+B_0 G_0$ and using the
second equation in \rfb{Halmos}, we get the first equation in
\rfb{Halmos}. Combining this with \rfb{Output1}, we get the third
equation in \rfb{Halmos}.

To show that \rfb{Halmos} implies \rfb{Inclusion1}, substitute $L_0=
A_0-B_0G_0$ in the first equation in \rfb{Halmos} to obtain
\vspace{-2mm} 
\[ \ddot{z}(t) \m=\m -A_0 z(t)-\Nscr_c(\dot{z}(t)) + B_0 G_0 \left[
   z(t) + A_0^{-1}\Nscr_c(\dot{z}(t))\right]. \]
Substituting the second equation in \rfb{Halmos} in the above
equation, we get \rfb{Inclusion1} (with an equality). Finally,
subtracting the second equation in \rfb{Halmos} from the third, we get
\rfb{Output1}.

So far, we have allowed any $n\in\nline$. If $n\leq 4$, then due to
Lemma \ref{Sobolev_embedding} we have $\Dscr(\Nscr_c)=H_\half$, so
that $\Nscr_c$ is an everywhere defined monotone and hemi-continuous
operator. According to Proposition \ref{Hemicont}, $\Nscr_c$ is maximal
monotone from $H_\half$ to $H_{-\half}$. Thus, we can apply Theorem
\ref{Main} to get the remaining statements in this
proposition. $\quad\square$

\smallskip

\begin{remark} \label{wave_generalized}
Using the same argument as in Remark \ref{Solutions}, we obtain that
for $n\leq 4$, for any input $u\in\Uscr$ and any initial state 
$\sbm{z_0 & w_0}^\top\in H_\half\times H$, the equation 
\rfb{Halmos} has a unique generalized solution $(\sbm{z\\ \dot z},u,
y)$. This satisfies $z\in C^1_r([0,\infty);H)\cap C([0,\infty);
H_\half)$ and \rfb{energy_bal_M}.
\end{remark}

\begin{remark}\label{n>4}
We comment on the case $n\geq 5$. In this case, $\Dscr(\Nscr_c)$ is
strictly included in $\Hscr_{\Gamma_0}^1(\Om)$ and we do not know if
$\Nscr_c$ is maximal monotone. Like any monotone operator, $\Nscr_c$
has maximal monotone extensions. One way to construct a maximal
monotone extension $\Nscr$, is indicated in \cite[p.~224]{Barbu}. We
can define the functional
\[ \varphi(w) \m=\m \frac{1}{4}\int_\Om w^4(\xi)\dd \xi\quad \FORALL 
   w\in \Hscr_{\Gamma_0}^1(\Om),\]
which may also be $+\infty$, and then $\Nscr$ is the subdifferential
of $\varphi$. It is not clear to us if this gives a genuine extension
of the cubic map $\Nscr_c$, or perhaps $\Nscr=\Nscr_c\m$ ?
\end{remark}

{\color{blue} \begin{proposition} \label{Airmax}
For $n\leq 3$, if $u\in\Hscr^1((0,\infty);U)$, $z(t),\dot{z}(t)\in
H_\half$ and the compatibility condition \rfb{Ben} (equivalently,
\rfb{Biden} or \rfb{Hila}) holds for $z_0=z(t)$ and $w_0=\dot{z}(t)$,
then the system of equations \rfb{Halmos} can be rewritten 
equivalently as follows\m{\rm :} \vspace{-1mm}
\begin{equation} \label{wave_nl}
   \left\{ \begin{array}{ll} \ddot z(x,t) \m=\m (\Delta z)(x,t) - 
   (\dot{z}(x,t))^3 \qquad  & \mbox{on }\Om\times
   [0,\infty),\\ z(x,t) \m=\m 0
   \bluff & \mbox {on }\Gamma_0\times [0,\infty), \\ \frac{\partial}
   {\partial \nu} \left[ z(\cdot,t) + A_0^{-1}\dot z(\cdot,t)^3
   \right](x) + b(x)^2 \dot{z}(x,t) \m=\m \sqrt{2} b(x) u(x,t)\bluff
   & \mbox{on }\Gamma_1 \times [0,\infty),\\ \frac{\partial}{\partial 
   \nu} \left[z(\cdot,t) + A_0^{-1}\dot z(\cdot,t)^3 \right](x) - 
   b(x)^2 \dot{z}(x,t) \m=\m \sqrt{2} b(x) y(x,t)\bluff & \mbox{on}\ 
   \Gamma_1 \times [0,\infty). \end{array} \right. \hspace{-2mm}\m 
\end{equation}
\end{proposition}}

In the above proposition (as in \rfb{wave}) we have written $\frac
{\partial}{\partial\nu}$ in place of $\gamma_1$, to make the formula 
more intuitive, and we have written $z(\cdot,t)$ in place of $z(t)$.
The second equation in \rfb{wave_nl} is not really needed if it is 
known that $z(t),\m\dot z(t)\in H_\half=\Hscr^1_{\Gamma_0}(\Om)$, but
we have written it in keeping with the PDE tradition that all the 
boundary conditions should be written explicitly.

{\em Proof.} \m For $n\leq 3$, it follows from Lemma
\ref{Sobolev_embedding} that $\Dscr(\Nscr_c)=\Hscr^1_{\Gamma_0}(\Om)
\subset L^6(\Om)$. Thus, for any $\dot{z}(t)\in \Dscr(\Nscr_c)$ we
obtain that $\Nscr_c(\dot{z}(t))\in L^2(\Omega)=H$. Using the fact
that $A_0$ is a boundedly invertible operator on $H$, we get that
$A_0^{-1}\Nscr_c(\dot{z}(t))\in \Dscr(A_0)=H_1\subset Z_0$. By the
compatibility condition \rfb{Biden}, we have that $z(t)+A_0^{-1}
\Nscr_c(\dot{z}(t))\in Z_0$, which implies that $z(t)\in Z_0$.
Therefore, the first line of \rfb{Halmos} can be rewritten as follows:
\vspace{-1mm}
\[ \ddot{z}(t) \m=\m -L_0 z(t)-L_0 A_0^{-1} \Nscr_c(\dot{z}(t)) .
   \vspace{-2mm} \]
On substituting $L_0=A_0-B_0 G_0$ in the second term on the 
right-hand side, we get \vspace{-1mm}
\[ \ddot{z}(t) \m=\m -L_0 z(t)-\left(A_0-B_0 G_0\right)A_0^{-1} 
   \Nscr_c(\dot{z}(t))=-L_0 z(t)-\Nscr_c(\dot{z}(t)), \vspace{-2mm}\]
because $A_0^{-1} \Nscr_c(\dot{z}(t))\in H_1$ and $G_0 H_1=\{0\}$. On
the space $Z_0$, we have $L_0=-\Delta$ and for $\dot z(t)\in H_\half$
we have $\Nscr_c(\dot{z}(t))=(\dot{z}(t))^3$. Thus, the first line of  
\rfb{Halmos} becomes the first line of \rfb{wave_nl}. Using the 
definition of $G_0$ and $C_0$ (as given in \rfb{Gorbachov}), the 
third and the fourth line of \rfb{wave_nl} can be easily derived
from the second and the third line of \rfb{Halmos}, respectively.
$\quad \square$

\smallskip

{\bf Wave equation with monotone boundary damping.} \m Next we
consider the wave equation with a nonlinear damping term acting on the
boundary $\Gamma_1$. Our wave equation system (a nonlinear
modification of \rfb{wave}) is\m:
\begin{equation} \label{waveBC}
   \resizebox{1.0\hsize}{!}{$\left\{ \begin{array}{ll} \ddot z(x,t) 
   \m=\m \Delta z(x,t) & \mbox{on }\Om\times[0,\infty),\\ z(x,t) \m=\m
   0 \bluff & \mbox{on }\Gamma_0\times [0,\infty), \\ \frac{\partial}
   {\partial \nu} z(x,t) + b(x)^2\dot{z}(x,t)+\sqrt{2}b(x)
   \Nscr_0\left(\sqrt{2} b(x)\dot{z}(x,t)\right) \m=\m \sqrt{2} b(x) 
   u(x,t) \bluff & \mbox{on }\Gamma_1 \times [0,\infty),\\ \frac
   {\partial }{\partial \nu} z(x,t) - b(x)^2\dot{z}(x,t) + \sqrt{2}b
   (x)\Nscr_0\left(\sqrt{2} b(x)\dot{z}(x,t)\right) \m=\m
   \sqrt{2}b(x)y(x,t)\bluff & \mbox{on }\Gamma_1 \times [0,\infty),\\
   z(x,0) \m=\m z_0(x),\ \ \dot z(x,0) \m=\m w_0(x) \bluff & 
   \mbox{on } \Om. \end{array} \right.$}\vspace{-2mm}
\end{equation}
In \rfb{waveBC}, the nonlinear boundary operator $\Nscr_0:\Dscr
(\Nscr_0)=L^2(\Gamma_1)\rarrow L^2(\Gamma_1)$ is the pointwise
application of a monotone Lipschitz continuous function $\beta:\rline
\rarrow \rline$, followed by multiplication with $\alpha$, where
$\alpha\in L^\infty(\Gamma_1)$, $\alpha(x)\geq 0$\m:
\[ \left(\Nscr_0(u)\right)(x) \m=\m \alpha(x)\beta\left(u(x)\right)
   \quad \mbox{ on}\ \ x\in\Gamma_1.\vspace{-1mm} \]

Recall the space $Z_0$ from \rfb{Z_0}. We see from Proposition
\ref{BCS_operators} that all the assumptions of Proposition \ref{CR7}
are satisfied and $ L_0 z = -\Delta z$ for all $z\in Z_0$. Therefore,
the nonlinear system \rfb{waveBC} fits into the abstract framework of
the systems with boundary control and boundary observation described
by \rfb{BSC}. By invoking Corollary \ref{WellposedBSC}, we obtain that
the equations \rfb{waveBC} determine an incrementally scattering
passive well-posed system. The classical solutions of this system
satisfy $z\in C^1_r([0,\infty); \Hscr^1_{\Gamma_0}(\Om))$ and
$\dot{z}\in C^1_r([0,\infty);L^2(\Om))$.

%%%%%%%%%%**********%%%%%%%%%%**********%%%%%%%%%%**********%%%%%%%%%%
\bibliographystyle{plain}        % Include this if you use bibtex 
%\bibliography{bibfile}  

\end{document}

%% file: ex_shared.tex
\usepackage{lipsum}
\usepackage{amsfonts}
\usepackage{graphicx}
\usepackage{epstopdf}
\usepackage{algorithmic}
\ifpdf
  \DeclareGraphicsExtensions{.eps,.pdf,.png,.jpg}
\else
  \DeclareGraphicsExtensions{.eps}
\fi

\newsiamremark{remark}{Remark}
\newsiamremark{hypothesis}{Hypothesis}
\crefname{hypothesis}{Hypothesis}{Hypotheses}
\newsiamthm{claim}{Claim}

\headers{Second order systems on Hilbert spaces with nonlinear 
  damping}{S. Singh and G. Weiss}

\title{Second order systems on Hilbert spaces \\
with nonlinear damping
\thanks{This research has been funded by the Israel Science 
Foundation under grant no. 3621/21.}}
\author{Shantanu Singh and George Weiss\thanks{School of Electrical 
  Engineering, Tel Aviv University, Ramat Aviv 69978, Israel. 
  \email{shantanu@tauex.tau.ac.il},
  \email{gweiss@tauex.tau.ac.il}.
}}%\footnotemark[3]

\usepackage{amsopn}